\documentclass[11pt]{article}
\usepackage{amssymb,amsmath,amsthm,amsfonts, commath, mathrsfs, mathtools, secdot}
\usepackage[T1]{fontenc}		
\usepackage[a4paper, margin=1in]{geometry}

\usepackage{bbm}
\usepackage{extarrows}
\usepackage{enumerate}
\usepackage{multicol}	
\usepackage{fge}
\usepackage{oldgerm}
\usepackage{hyperref} 
\usepackage{paralist} 
\usepackage{accents}

\usepackage{tikz}
\usetikzlibrary{datavisualization}
\usetikzlibrary{datavisualization.formats.functions}
\usepackage{pgfplots}
\usepackage{tkz-fct}
\pgfplotsset{every axis/.append style={
    axis x line=middle,    
    axis y line=middle,    
    axis line style={<->}, 
    xlabel={$x$},          
    ylabel={$y$},          
    },
    cmhplot/.style={color=blue,mark=none,line width=1pt,<->},
    soldot/.style={color=blue,only marks,mark=*},
    holdot/.style={color=blue,fill=white,only marks,mark=*},
}
    
\tikzset{>=stealth}

\usepackage[natbib=true, backend=biber, style=numeric]{biblatex}
\usepackage{csquotes}
\sectiondot{subsection}
\usepackage{xcolor}

\usepackage{color}

\usepackage[]{stackengine}

\newlength\correct
\newtheorem{theorem}{Theorem}[section]
\newtheorem{lem}[theorem]{Lemma}
\newtheorem{prop}[theorem]{Proposition}
\newtheorem{cor}[theorem]{Corollary}
\theoremstyle{definition}\newtheorem{example}[theorem]{Example}
\theoremstyle{definition}
\theoremstyle{definition}\newtheorem{remark}[theorem]{Remark}
\theoremstyle{definition}\newtheorem{defi}[theorem]{Definition}

\newtheorem{innerCondition}{Condition}
\newenvironment{Condition}[1]
  {\renewcommand\theinnerCondition{#1}\innerCondition}
  {\endinnerCondition}

\newtheorem{innerAssumption}{Assumption}
\newenvironment{Assumption}[1]
  {\renewcommand\theinnerAssumption{#1}\innerAssumption}
  {\endinnerAssumption}

\numberwithin{equation}{section}

\newcommand{\R}{\mathbb{R}}

\newcommand*\diff{\mathop{}\!\mathrm{d}}								
\newcommand{\E}{\mathbb{E}}											
\newcommand{\Norm}{\mathcal{N}}									         
\newcommand{\Pro}{\mathbb{P}}										
\newcommand{\indep}{\perp \!\!\! \perp}									
\newcommand{\oland}{\mbox{\scriptsize $\mathcal{O}$}}						
\newcommand{\Oland}{\mathcal{O}}										
\newcommand{\ind}{\operatorname{\mathbf{1}}}							
\renewcommand{\max}[1]{\underset{#1}{\operatorname{max}}\;}				
\renewcommand{\sup}{\operatorname{sup}}					
\renewcommand{\|}{\; : \, \;}											
\newcommand{\conv}[2]{\; \xrightarrow[ {#1}]{#2} \; }						
    
\newcommand{\littletaller}{\mathchoice{\vphantom{\big|}}{}{}{}}
\newcommand\restr[2]{{
  \left.\kern-\nulldelimiterspace 
  #1 
  \littletaller 
  \right|_{#2} 
  }}

\makeatletter
    \def\@fnsymbol#1{\ensuremath{\ifcase#1\or \ddagger\or  \mathsection \or \dagger \or \mathparagraph\or \|\or **\or \dagger\dagger
        \or \ddagger\ddagger \else\@ctrerr\fi}}
\makeatother

\newcommand*\samethanks[1][\value{footnote}]{\footnotemark[#1]}

\title{A functional central limit theorem for particle systems under polynomial growth conditions, with applications to statistics}
\author{}
\date{}

\author{Mark Podolskij\thanks{Department of Mathematics, University of Luxembourg} \, and Fabrice Wunderlich\samethanks \,
\thanks{Department of Statistics, London School of Economics}}

\begin{document}
\maketitle

\begin{abstract}
We establish a functional central limit theorem (CLT) for fluctuations of time-integrated functionals of empirical measures arising from interacting McKean--Vlasov particle systems. Our framework allows for coefficients with unbounded spatial growth and dependence on arbitrarily high-order moments of the measure variable, extending the scope of the existing fluctuation theory. The proof uses a novel approach combining linear functional differentiation with new well-posedness results and estimates for perturbed backward parabolic equations in polynomially weighted Schauder spaces, yielding a martingale-type representation of the fluctuations with an asymptotically negligible remainder. The limiting process is shown to be centred Gaussian with an explicit covariance structure. As a statistical application, we derive the asymptotic theory underlying hypothesis tests for both the temporal constancy of population characteristics and parametric restrictions on the drift coefficient.
\end{abstract}

\section{Introduction}
\label{sec:introduction}

Interacting particle systems of McKean--Vlasov type provide a probabilistic approximation of nonlinear diffusion equations whose coefficients depend not only on the current state of the process but also on its law (see, for instance, the classical references \cite{McKean,Sznitman,Meleard}). Throughout this paper, we consider a system of $N$ weakly interacting particles
\begin{align}
\begin{cases}
    \diff X_t^{i,N}
    =
    b\big(X_t^{i,N},\mu_t^N\big)\diff t
    +
    a\big(X_t^{i,N},\mu_t^N\big)\diff W_t^i,
    \qquad i=1,\ldots,N, \\[0.15cm]
    \mathcal L(X_0^{1,N},\ldots,X_0^{N,N})=\mu_0^{\otimes N},
\end{cases}
\tag{IPS}\label{eq:IPS}
\end{align}
defined on a filtered probability space
$(\Omega,\mathcal F,(\mathcal F_t)_{t\in [0,T]},\Pro)$ satisfying the usual conditions. Here,
$W^1,\ldots,W^N$ are independent one-dimensional Brownian motions and
$$
\mu_t^N:=\frac1N\sum_{i=1}^N\delta_{X_t^{i,N}}
$$
denotes the empirical measure of the particle system at time $t$. Under standard assumptions, such as suitable Lipschitz continuity of the coefficients $b$ and $a$ alongside sufficient moment bounds, the particle system \eqref{eq:IPS} exhibits a mean-field limit described by the McKean--Vlasov equation
\begin{align}
\begin{cases}
\diff X_t
=
b(X_t,\mu_t)\diff t
+
a(X_t,\mu_t)\diff W_t,
\quad \mu_t=\mathcal L(X_t), \\[0.15cm]
\mu_{t=0}=\mu_0.
\end{cases}
\tag{MKV}\label{eq:MKV}
\end{align}
In this paper, we investigate the fluctuation behaviour of integrated empirical functionals of the form
\begin{align}
t \mapsto \int_0^t \int_{\R} \phi(s,x) \mu_s^N(\diff x)
\diff s
\label{eq:empirical_functional}
\end{align}
on the space of continuous paths over a compact time interval, around their McKean--Vlasov counterparts
\begin{align}
t \mapsto \int_0^t \int_{\R} \phi(s,x) \mu_s(\diff x)
\diff s ,
\label{eq:McKean_Vlasov_functional}
\end{align}
where $\phi$ belongs to a large class of possibly unbounded functions specified below. A distinctive feature of our framework is that the coefficients $a$ and $b$ are merely required to be Lipschitz continuous with respect to the Wasserstein-$p$ distance in the measure variable, for an arbitrary $p \ge 2$. This allows for coefficients that grow linearly in the spatial variable and depend on moments of the measure of arbitrarily high order, a setting that, as discussed below, lies beyond the scope of the existing literature.

A natural motivation for allowing a general dependence on the empirical measure comes from mean-field models in which agents respond to aggregate characteristics of the population, as in opinion dynamics, consensus formation, flocking, and related multi-agent systems (see, e.g., \cite{HegselmannKrause2002,CuckerSmale2007,MotschTadmor2014,PareschiToscani2013}). Classical models of this type often depend on simple (weighted) averages or other first-order population statistics. However, first-order summaries do not capture features such as dispersion, asymmetry, concentration, or the prevalence of extreme population states, which are all naturally reflected in higher-order characteristics of the empirical distribution.

At the law-of-large-numbers level, convergence of the integrated empirical functionals follows from standard propagation-of-chaos arguments (see \cite{Sznitman,Meleard} for the classical $p=2$ setting and Section \ref{sec:appendix_propagation_of_chaos_Wp} in the \hyperref[appn]{Appendix} for the extension to $p>2$). More precisely, by coupling the interacting particle system with an independent McKean--Vlasov shadow system through synchronous coupling and combining this with the classical Wasserstein estimates of \cite[Thm.~1]{FournierGuillin}, one obtains an $N^{-1/2}$-rate for the expected supremum over compact time intervals of the $p$-th power of the Wasserstein-$p$ distance between $\mu_t^N$ and $\mu_t$.

Obtaining fluctuation results at the central-limit scale is substantially more delicate. Pure Wasserstein estimates derived from propagation of chaos are generally insufficient for this purpose. Although sharper Wasserstein bounds have been obtained under additional structural assumptions (see, e.g., \cite{AmorinoEtAl}), these approaches rely on strong confinement or convexity conditions. More fundamentally, refined asymptotic results for empirical Wasserstein distances, even in the i.i.d. setting, require tail-sensitive regularity assumptions on the limiting distribution, such as the finiteness of a certain quantile functional \cite{delBarrioGineMatran1999}. Such conditions are not implied by the standard assumptions for McKean--Vlasov systems and may fail even for elliptic diffusions with Gaussian-type tails.

Nevertheless, there exists a rich literature on fluctuation CLTs for McKean--Vlasov systems. Classical approaches investigate the empirical fluctuation field
$$
t \, \mapsto \, \sqrt N(\mu_t^N-\mu_t)
$$
as a process taking values in distribution spaces or negative Sobolev spaces, thereby also covering, in principle, the integrated functionals \eqref{eq:empirical_functional}--\eqref{eq:McKean_Vlasov_functional}. The coefficient classes covered by these results, however, are substantially more restrictive than the one considered here. Indeed, foundational contributions by Hitsuda--Mitoma \cite{HitsudaMitoma1986} and Fernandez--Méléard \cite{FernandezMeleard1997} establish convergence of such fluctuation fields for interacting diffusions with convolution-type coefficients of the form
$\tilde b(x,\nu)=\int_\R b(x,y)\nu(\diff y)$.
While these results are highly influential, they require, among other assumptions, bounded partial derivatives of the coefficient kernel up to order three. Consequently, they do not cover coefficients with unbounded spatial growth or high-order polynomial dependence on the measure.

More recently, Jourdain--Tse \cite{JourdainTse2021} established CLTs for nonlinear functionals of empirical measures and applied them to mean-field fluctuations of interacting diffusions, while Delarue--Lacker--Ramanan \cite{DelarueLackerRamanan2020} derived a functional CLT through the master-equation approach in mean field games, building on the analytic framework of \cite{CardaliaguetDelarueLasryLions2019}. Despite their different methodologies, these works continue to rely on strong boundedness assumptions both in the spatial variable and for derivatives with respect to the measure variable. As in the classical theory, these conditions are not suited to the high-order polynomial moment dependence allowed in the present framework.

\subsection{{Main contribution}}
The main contribution of the present paper is the development of a fluctuation CLT framework for integrated functionals of interacting McKean--Vlasov particle systems whose coefficients may be unbounded in the spatial variable and may depend on high generalised moments of the measure variable. To the best of our knowledge, no existing work in the literature covers this level of generality. Consequently, our setting falls outside the scope of the classical fluctuation CLT theory for McKean--Vlasov particle systems.

The proof of the CLT relies on two main ingredients. First, we exploit the concept of linear functional differentiability of the coefficients with respect to the measure variable. This framework, which has become standard in recent years
(see, e.g., \cite{JourdainTse2021, DelarueLackerRamanan2020, CardaliaguetDelarueLasryLions2019}), enables us to linearise the dynamics with respect to the measure argument. We refer to Section~\ref{subsec:linear_functional_derivatives} for a detailed discussion. Second, we develop a novel approach based on backward parabolic partial differential equations (PDEs), which allows us to propagate polynomial growth of the coefficients to the corresponding Schauder solutions of extended backward equations. Since the fluctuation analysis cannot be carried out within the classical topology of bounded test functions, we instead work with polynomially weighted test-function spaces and establish suitable estimates for the associated McKean--Vlasov backward equations. These estimates ultimately yield a martingale-type representation for the fluctuation of the integrated functional
\begin{align}
t \; \mapsto \;  \sqrt N\int_0^t \int_{\R}
 \phi(s,x)(\mu_s^N-\mu_s)(\diff x)
\diff s \label{eq:fluctuation_functional_target}
\end{align}
up to a remainder that is negligible at the central-limit scale. Combining this decomposition with a martingale CLT and a functional tightness argument, we identify the limiting process as a Gaussian process with an explicit covariance structure.

Our probabilistic results are strongly motivated by statistical applications. To illustrate different aspects of the fluctuation theory developed in this paper, we consider two testing problems. The first concerns the temporal constancy of macroscopic population characteristics of the form
\[
{\bf m}_t = \left( \int_{\R} \varphi_1(x)\,
\mu_t(\diff x),\ldots,\int_{\R} \varphi_q(x)\,
\mu_t(\diff x) \right), \qquad t\in [0,T],
\]
for given functions $\varphi_1,\ldots,\varphi_q$. Such quantities arise naturally when the coefficients depend on the law only through finitely many moments or other aggregate characteristics, in which case they encode the mean-field feedback acting on each particle.
Second, we consider parametric hypothesis testing for the drift. More precisely, for given functions
$b_1,\ldots,b_d:\R\times\mathcal P_p(\R)\to\R$, we introduce the finite-dimensional model
$$V := \left\{b_\theta=\sum_{\ell=1}^d\theta_\ell b_\ell:\theta=(\theta_1,\ldots,\theta_d)^\top\in\R^d \right\},$$
and we test, under the assumption that \(b=b_{\theta^*}\) for some unknown \(\theta^*\in\R^d\), the linear restrictions 
$$ \Xi\theta^*=\xi.$$
A precise formulation of both testing problems is provided in Section~\ref{sec:Stat}. In each case, the asymptotic distribution of the corresponding test statistic is obtained from the fluctuation results developed in this paper.

\subsection{Structure of paper}
Section \ref{sec:setup} presents the main setup of the paper, the key assumptions, and the definition of the linear functional derivative with respect to the measure component. Section \ref{subsec:technicalities} is devoted to certain parabolic backward PDEs and polynomial estimates for their solutions. Section \ref{sec:MainStat} presents the main results of the paper, in particular establishing functional CLTs for fluctuations of the form \eqref{eq:fluctuation_functional_target} using the PDE tools established in Section \ref{subsec:technicalities}. Statistical applications to testing for the drift component are discussed in Section \ref{sec:Stat}. Proofs of the main results are collected in section \ref{sec:Proofs}.

\subsection{Notation}
\label{subsec:notation}

For two sequences $(a_N)_{N\ge 1}$ and $(b_N)_{N\ge 1}$, we write $a_N\lesssim b_N$ if $a_N\leq Cb_N$ for some $C>0$.

Let $U\subseteq[0,T]\times\R$ be an open set. For integers
$\gamma,\lambda\ge0$, let $\mathscr C^{\gamma,\lambda}(U)$ denote the
space of functions $f:[0,T]\times\R\to\R$ that are $\gamma$ times
continuously differentiable with respect to the time variable and
$\lambda$ times continuously differentiable with respect to the spatial
variable on $U$. Further, if all derivatives are bounded on $U$, we will denote this space by $\mathscr C_b^{\gamma,\lambda}(U)$ and endow it with the norm
\[
 \norm{f}_{\mathscr{C}^{\gamma,\lambda}_b(U)} \; := \; \sum_{0\le \beta\le \gamma } \operatorname{sup}_{(t,x) \in U} |\partial_t^{\beta} f(t,x)| \; + \; \sum_{0\le \beta\le \lambda } \operatorname{sup}_{(t,x) \in U} |\partial_x^{\beta} f(t,x)|.
\]
Fix $\alpha\in(0,1)$ throughout, and let
$\mathscr H^{\gamma+\alpha/2,\lambda+\alpha}(U)$ denote the anisotropic Hölder space  consisting of all
$f\in\mathscr C_b^{\gamma,\lambda}(U)$ satisfying
$$ 
\norm{f}_{\mathscr{H}^{\gamma+\alpha/2,\lambda +\alpha}(U)} \; := \; \norm{f}_{\mathscr{C}^{\gamma,\lambda}_b(U)} \; + \; [\partial^{\gamma}_t f]_{\alpha} \; + \; [\partial^{\lambda}_x f]_{\alpha} \; < \; \infty,
$$
where
$$ 
[f]_{\alpha} \; := \; \operatorname{sup}_{\stackrel{(t,x)\neq (s,y)}{(t,x),(s,y) \in U}} \frac{|f(t,x)-f(s,y)|}{(|t-s|^{1/2} + |x-y|)^\alpha} .
$$
In particular, no mixed derivatives are included. 
If $f$ depends only on the spatial variable, we simply write
$f\in\mathscr C_b^\lambda(U)$ and
$f\in\mathscr H^{\lambda+\alpha}(U)$, respectively. 

For $r\ge 1$, we use the notation  
$$P(r)(x):= (1+|x|^2)^{r/2}.$$ 
We further introduce  the polynomially weighted Banach spaces 
$$ \mathscr{C}^{\gamma,\lambda}_{P(r)} \; := \; \mathscr{C}^{\gamma,\lambda}_{P(r)}([0,T]\times \R) \; := \; \big\{ f \; : \; fP(r)^{-1} \in \mathscr{C}^{\gamma,\lambda}_{b}([0,T]\times \R)  \big \}$$ 
and 
$$ \mathscr{H}^{\gamma+\alpha/2,\lambda +\alpha}_{P(r)} \; := \; \mathscr{H}^{\gamma+\alpha/2,\lambda +\alpha}_{P(r)}([0,T]\times \R) \; := \; \big\{ f \; : \; fP(r)^{-1} \in \mathscr{H}^{\gamma+\alpha/2,\lambda +\alpha}([0,T]\times \R)  \big \}.$$
as well as
$$ \widetilde{\mathscr{H}}^{1,2+\alpha}_{P(r)} \; := \; \Big\{ f \in \mathscr{C}^{1,2} \; : \; \norm{f}_{\widetilde{\mathscr{H}}^{1,2+\alpha}_{P(r)}} \; := \; \operatorname{sup}_{0\le t \le T} \norm{f(t,\cdot)}_{\mathscr{H}^{2+\alpha}_{P(r)}} \, + \, \operatorname{sup}_{\substack{0\le t\le T\\ x \in \R}} \Big| \frac{\partial_t f(t,x)}{P(r+2)(x)}\Big|\; < \; \infty \Big\},$$
and, for $I:=[a,b]$, $0\le a<b\le T$,
$$\widetilde{\mathscr{H}}^{1,2+\alpha}_{P(r)}(I\times \R):= \;  \Big\{ f|_{I\times \R}: \, f \in \widetilde{\mathscr{H}}^{1,2+\alpha}_{P(r)}\Big\}.$$

Finally, for a probability measure $\nu$ on $\R$ and a function
$h\in L^1(\R,\nu)$, we use the pairing notation
\[
\langle h,\nu\rangle
:=
\int_{\R} h(x)\,\nu(\diff x).
\]
For $q\geq 1$, we denote by
\[
\mathcal{P}_q(\R)
:=
\left\{
\nu\in\mathcal{P}(\R):
\int_{\R}|x|^q\,\nu(\diff x)<\infty
\right\}
\]
the space of Borel probability measures on $\R$ with finite $q$-th moment; the $q$-th absolute moment of $\nu$ is denoted by $\mathcal{M}_q(\nu)$. This space is equipped with the Wasserstein-$q$ distance
$$ 
\mathcal{W}_q(\mu,\nu) := \Big( \operatorname{inf}_{\pi\in\Gamma(\mu,\nu)} \int_{\R^2}|x-y|^q\,\pi(\diff x,\diff y) \Big)^{1/q} , 
$$
where $\Gamma(\mu,\nu)$ denotes the collection of all couplings of $\mu$ and $\nu$, i.e. all Borel probability measures on $\R^2$ with marginals $\mu$ and $\nu$. 

Lastly, for a signed measure $\nu$, we write $|\nu|=\nu^{+}+\nu^{-}$ to denote the total variation measure of $\nu$.



\section{General setting and linear functional derivatives}

\subsection{Core setup}
\label{sec:setup}

In this section, we state the core assumptions on the drift and diffusion coefficients, which are maintained across all subsequent sections without explicit mention.
We fix $p\ge 2$.

\begin{Assumption}{M1} \label{ass:Main}
The coefficients
$a,b:\R\times\mathcal P_p(\R)\to\R$
are Borel measurable, and the diffusion coefficient is uniformly elliptic, that is,
$$
a^2(x,\nu)\ge \lambda>0,
\qquad (x,\nu)\in\R\times\mathcal P_p(\R),
$$
for some constant $\lambda>0$.
Furthermore, for all $(x,\nu),(\tilde x,\tilde\nu)\in\R\times\mathcal P_p(\R)$,
\begin{align}
 |b(x,\nu)-b(\tilde x,\tilde\nu)|
 +
 |a(x,\nu)-a(\tilde x,\tilde\nu)|
 \le
 C  \big(|x-\tilde x|+\mathcal{W}_p(\nu,\tilde\nu)\big).
\label{eq:Lipschitz_continuity_of_coefficients}
\end{align}
\end{Assumption}

The Lipschitz continuity condition \eqref{eq:Lipschitz_continuity_of_coefficients} immediately implies the linear growth estimate
\begin{align}
|b(x,\mu)|+|a(x,\mu)|
\; \leq \; C \bigl(1+|x|+\mathcal{M}_p(\mu)^{1/p}\bigr), \qquad x \in \R, \, \mu \in \mathcal{P}_p(\R),
\label{eq:linear_growth}
\end{align}
for some constant $C>0$.

The particle system \eqref{eq:IPS} satisfies a propagation-of-chaos property that is sufficiently strong to imply convergence of the empirical measures towards the McKean--Vlasov law in the Wasserstein-$p$ distance, uniformly over compact time intervals. Such results are well established under global Lipschitz assumptions in the classical Wasserstein-$2$ setting (see, e.g., \cite{Sznitman,Meleard}). Since our assumptions are formulated in terms of the Wasserstein-$p$ distance, we include, for completeness, a self-contained proof of the corresponding result for $p\ge2$ in Proposition~\ref{prop:propagation_chaos_Wp} of Section~\ref{sec:appendix_propagation_of_chaos_Wp} in the \hyperref[appn]{Appendix}.

\begin{Assumption}{M2} 
The initial law satisfies the moment bound $\mathcal{M}_{2p+3}(\mu_0)<\infty$. 
\end{Assumption}
Under this moment bound, the solution of the McKean--Vlasov SDE \eqref{eq:MKV} and the interacting particle system \eqref{eq:IPS} propagate this moment uniformly over compact time intervals and in $N\ge1$ (see Proposition~\ref{prop:Wp_well_posedness_moment_propagation} and Remark~\ref{rem:moment_propagation}). More precisely,
\begin{align}
    \E \big[ \operatorname{sup}_{0\le t \le T} |X_t|^{2p+3}\big]  \; + \; \operatorname{sup}_{N\ge1} \E\big[\operatorname{sup}_{0\le t\le T}  \, \mathcal{M}_{2p+3}(\mu^N_t) \big] \, <\, \infty. \label{eq:moments_of_mu}
\end{align}
Moreover, the linear growth condition \eqref{eq:linear_growth} yields
$$
\operatorname{sup}_{N\ge1}
\Big(
\E\big[\operatorname{sup}_{0\le t\le T} |b(X_t^{1,N},\mu_t^N)|^{2p+3}\big]
\vee
\E\big[\operatorname{sup}_{0\le t\le T} |a(X_t^{1,N},\mu_t^N)|^{2p+3}\big]
\Big)
<\infty.
$$

We conclude this setup by introducing a last regularity condition on the coefficients, which is crucial for analysing the auxiliary PDEs in Section~\ref{subsec:technicalities}. 

\begin{Assumption}{M3} \label{condition:coefficient_regularity}
    The coefficients $\beta(t,x):=b(x,\mu_t)$, $\sigma(t,x):=a(x,\mu_t)$ satisfy, for all $0\le t \le T$,
    \begin{itemize}
        \item[(i)] $\partial_x \beta, \partial_x\sigma \in \mathscr{C}_b^{0,2}([0,T]\times \R)$, 
        \item[(ii)] for each $0\le k\le 3$ the functions $\partial_x^k \beta, \partial_x^k \sigma^2$ are jointly continuous on $[0,T]\times \R$, and 
        $$\operatorname{sup}_{0\le t\le T} \, \Big( \lVert \beta(t,\cdot)\rVert_{\mathscr{H}^{3+\alpha}([-R,R])} \, + \, \lVert\sigma^2(t,\cdot)\rVert_{\mathscr{H}^{3+\alpha}([-R,R])}\Big) \; < \; \infty, $$
        for each $R>0$.
    \end{itemize}
\end{Assumption}

As previously noted, Assumptions~\ref{ass:Main}--\ref{condition:coefficient_regularity}  hold throughout the remainder of this paper.

\subsection{Linear functional derivatives}
\label{subsec:linear_functional_derivatives}

To identify the limiting law of statistics of the form
\eqref{eq:fluctuation_functional_target}, it is necessary to understand the evolution of the fluctuation field
\[
t\mapsto \eta_t^N:=\sqrt{N}(\mu_t^N-\mu_t).
\]
A standard approach consists of applying Itô's formula separately to
\eqref{eq:empirical_functional} and \eqref{eq:McKean_Vlasov_functional}, and then analysing the difference between the resulting identities. This naturally leads to the problem of estimating terms of the form
$b(x,\mu_t^N)-b(x,\mu_t)$ and
$a(x,\mu_t^N)-a(x,\mu_t)$.

A widely adopted approach to this problem in recent years has been to assume an appropriate notion of differentiability of the coefficients with respect to the measure variable (see, for example,
\cite{JourdainTse2021, DelarueLackerRamanan2020, CardaliaguetDelarueLasryLions2019}). While this assumption is typically not restrictive, it provides a convenient framework for linearising the dynamics. Heuristically, the leading-order contribution of the measure dependence is obtained by performing a first-order expansion around the limiting measure $\mu_t$. More explicitly, one expects
$$
b(x,\mu_t^N)-b(x,\mu_t) \; = \; \big\langle D_\mu b(x,\mu_t),\mu_t^N-\mu_t\big\rangle
\, + \,
\operatorname{Remainder}(b,x,t,N).
$$
After multiplication by $\sqrt{N}$, the remainder term should become asymptotically negligible. Consequently, the nonlinear dependence of the McKean--Vlasov dynamics on the law is transformed, at the fluctuation scale, into a linear response driven by the fluctuation field.

The purpose of this section is to make this linearisation rigorous. We introduce the notion of linear functional derivatives for functions defined on $\mathcal P_p(\mathbb R)$ and formulate the differentiability assumptions imposed on the coefficients. These assumptions will subsequently allow us to control the remainder terms arising from the second-order Taylor expansion.

\begin{defi} \label{defi:functional_derivative}
 A map $\Phi:\mathcal P_p(\R)\to\R$ is said to admit a \emph{linear functional derivative at $\nu$}, or to be \emph{differentiable at $\nu$}, if there exists a measurable map $D_\nu\Phi(\nu):\R\to\R$ such that, for every $\tilde\nu\in\mathcal P_p(\R)$ one has $D_\nu\Phi(\nu)\in L^1\big(|\tilde\nu-\nu|\big)$ and
    \begin{align*}
        \Phi(\nu + \varepsilon (\tilde \nu - \nu)) - \Phi(\nu ) \; = \; \varepsilon \int_{\R} D_\nu \Phi(\nu)(y) \; (\tilde{\nu}-\nu)(\diff y) \, + \, \oland(\varepsilon)
    \end{align*}
    as $\varepsilon \downarrow 0$.
\end{defi}

In particular, if  $\Phi$ is differentiable along the segment $\nu+ \lambda(\tilde \nu - \nu)$, $0\le \lambda \le 1$, then this implies 
    \begin{align*}
        \Phi(\tilde \nu) - \Phi(\nu ) \; = \; \int_0^1 \int_{\R} D_\nu \Phi(\nu +\lambda(\tilde \nu- \nu ))(y) \; (\tilde\nu - \nu)(\diff y) \, \diff \lambda.
    \end{align*}
Moreover, if $\Phi$ is twice differentiable along the segment, then additionally there exists a map $D_{\nu\nu} \Phi(\nu): \R^2 \to \R$ with $D_{\nu\nu} \Phi(\nu) \in L^1(|\tilde \nu-\nu|^{\otimes 2}  )$, such that 
    \begin{align}
        &\Phi(\tilde \nu) - \Phi(\nu ) \, = \,  \notag \\
        &\int_{\R} D_\nu \Phi(\nu)(y) \, (\tilde \nu - \nu)(\diff y)  +  \int_0^1 (1-\lambda) \int_{\R^2} D_{\nu\nu} \Phi( \nu +\lambda(\tilde \nu- \nu ))(y,z)  \, (\tilde{\nu}-\nu)^{\otimes^2}(\diff y,\diff z)  \diff \lambda \label{eq:second_order_Lions_expansion}.
    \end{align}

In order to be able to make use of this concept, we will require the existence of the first and second order linear derivative for our coefficient functions along the segments $\mu_t+\lambda(\mu_t^N-\mu_t)$, $0\le \lambda \le 1$, $0\le t\le T$. The following regularity condition formalises this.
\begin{Condition}{R1} \label{ass:R}
      For every $x\in \R$ and $j=1,...,d$, the mappings $\nu \mapsto b(x,\nu)$, and $\nu \mapsto a(x,\nu)$ are twice differentiable along the random segments $\mu^{N,\lambda}_t:=\mu_t+\lambda(\mu_t^N-\mu_t)$, $0\le \lambda \le 1$, $0\le t\le T$, and their linear functional derivatives are suitably integrable, that is
      \begin{align*} 
      &\int_{\R^2} \, | D_\nu f(x,\mu^{N,\lambda}_t)(y)|^2 \, \mu^{N,\lambda}_t(\diff x) \,  |\mu^N_t-\mu_t|(\diff y) \; < \; \infty\\
      &\int_{\R^3} \, | D_{\nu \nu} f(x,\mu^{N,\lambda}_t)(y,z)|^2 \, \mu^{N,\lambda}_t(\diff x) \, |\mu^N_t-\mu_t|^{\otimes 2}(\diff y, \diff z) \; < \; \infty
      \end{align*}
      almost surely for $f=a$ and $f=b$, $0\le \lambda \le 1$..
\end{Condition}

\begin{remark} \label{rem:remark_to_condition_R1}
    We shall note that Condition \ref{ass:R} implies that the maps $\nu \mapsto \langle f(\cdot ,\nu) g(\cdot, \nu), \nu \rangle$, $f,g \in \{1, b,a\}$, are twice differentiable at all $\mu_t+\lambda(\mu^N_t-\mu_t)$, $0\le \lambda \le 1$, $0\le t\le T$.
\end{remark}

\begin{remark}
Due to the moment bounds \eqref{eq:moments_of_mu}, the integrability requirements in Condition~\ref{ass:R} can be verified directly under suitable polynomial growth conditions on the linear functional derivatives. For instance, it is sufficient to assume that, for $f\in\{a,b\}$,
\[
|D_\nu f(x,\nu)(y)|
\; \lesssim \; 
(1+|x|^{p+1})(1+|y|^{p+1}),
\]
and
\[
|D_{\nu\nu}f(x,\nu)(y,z)|
\; \lesssim \, 
\prod_{r\in\{x,y,z\}}(1+|r|^{p+1}),
\]
uniformly for $\nu=\mu_t+\lambda(\mu_t^N-\mu_t)$, $0\le t\le T$, and $0\le\lambda\le1$. Indeed, after squaring, these bounds involve moments of order at most $2p+2$ in each variable, which are controlled by the moment bound of order $2p+3$ in \eqref{eq:moments_of_mu}.
\end{remark}

\section{Parabolic backward PDEs and polynomial estimates of solutions} \label{subsec:technicalities}

The central fluctuation objects studied in this paper are the integrated observables of the empirical-measure fluctuation field introduced in \eqref{eq:fluctuation_functional_target}, where the class of admissible time-dependent test functions $\phi$ will be specified later on.

The purpose of this section is to develop the analytic framework for obtaining a martingale-type representation of the integrated empirical fluctuations. The underlying idea is as follows. Given a time-dependent test function $\phi$, we solve a backward parabolic equation with source term $\phi$ and subsequently apply Itô's formula to its solution evaluated along the particle trajectories. This yields a representation of \eqref{eq:fluctuation_functional_target} as the sum of the terminal value of an auxiliary martingale and higher-order error terms, with the latter being structurally simpler to control.

Unlike in existing classical settings, the coefficients considered here are allowed to exhibit polynomial growth, as are the admissible test functions. Consequently, the standard global Schauder theory is no longer sufficient. We therefore require new estimates showing that polynomial growth of the terminal condition and source term is propagated to the solution of the backward equation together with its derivatives. Such estimates are needed both for the standard backward parabolic equation \eqref{eq:backward_PDE_standard_global_Schauder}, associated with the McKean--Vlasov generator, and for an extended backward equation \eqref{eq:backward_PDE_with_A_operator_Schauder} introduced later in this section. To this end, we first recall the classical Schauder theory.

Consider the backward PDE
\begin{align}
\partial_t \varphi(t,x) \; + \;\mathcal{V}_{t,\mu_t} \varphi (x) \;-\; f(t,x) \; = \; 0,
\qquad
\varphi(T,x)=g(x),
\label{eq:backward_PDE_standard_global_Schauder}
\end{align}
where
\begin{align}
\mathcal{V}_{t,\nu} \varphi (x)
\; := \;
b(x,\nu) \partial_x \varphi(t,x)
+
\frac{a^2(x,\nu)}{2}
\partial_{xx}\varphi(t,x)
\label{eq:definition_of_operator_mathcal_V}
\end{align}
is the McKean--Vlasov generator.
Recall that $\beta(t,x)=b(x,\mu_t)$ and $\sigma(t,x)=a(x,\mu_t)$.
By the classical Schauder theory (see, e.g., the monograph by Lieberman \cite{Lieberman1996}), since $a^2$ is uniformly elliptic, if
$\beta,\sigma^2,f\in\mathscr{H}^{\alpha/2,\alpha}([0,T]\times\R)$
and
$g\in\mathscr{H}^{2+\alpha}(\R)$,
then \eqref{eq:backward_PDE_standard_global_Schauder} admits a unique global solution
$\varphi\in\mathscr{H}^{1+\alpha/2,2+\alpha}([0,T]\times\R)$.

\begin{remark}
    In fact, if $\beta, \sigma^2, f \in \mathscr{H}^{\alpha/2,\alpha}$ and $g\in \mathscr{H}^{2+\alpha}$ on compact subsets of $[0,T]\times \R$, this is enough for the backward PDE \eqref{eq:backward_PDE_standard_global_Schauder} to have a unique solution $\varphi$ such that $\varphi \in \mathscr{H}^{1+\alpha/2,2+\alpha}$ on compact subsets. 
\end{remark}

\begin{remark} \label{rem:coefficients_are_locally_Hölder}
    Based on Assumption \ref{ass:Main}, the coefficients $\beta$ and $\sigma^2$ are $1/2$-Hölder continuous in time and Lipschitz continuous in space.
\end{remark}

Polynomially weighted Schauder estimates for equations with unbounded
coefficients were obtained by Palczewski
\cite[Thm.~3.6 and Rem.~3.7]{palczewski}.
However, those results do not explicitly state whether the estimate
constant can be chosen uniformly when the backward problem
\eqref{eq:backward_PDE_standard_global_Schauder} is posed on compact
time subintervals of $[0,T]$. This uniformity is required to obtain uniqueness and, based on this, through a fixed point argument, existence of a solution to the extended backward PDE
\eqref{eq:backward_PDE_with_A_operator_Schauder}.
We therefore derive the weighted estimate below from
\cite[Thm.~3.8]{lorenzi} by conjugating with $P(r)$ and establish
the required uniformity by applying that theorem on a fixed time horizon.

\begin{theorem} \label{thm:propagation_of_polynomial_control}
Let $r\ge 1$ and suppose Assumptions \ref{ass:Main}--\ref{condition:coefficient_regularity}. Suppose $g \in \mathscr{H}^{2+\alpha}_{P(r)}$ and that $f \in \mathscr{C}([0,T]\times \R)$ with $\operatorname{sup}_{0\le t\le T}\lVert f(t,\cdot)\rVert_{\mathscr{H}^\alpha_{P(r)}}< \infty$. Then, there exists one and only one solution $\varphi \in \mathscr{C}^{1,2}$ to the backward PDE \eqref{eq:backward_PDE_standard_global_Schauder} for which it holds 
$$ \sup_{\substack{0\le t\le T\\ x\in \R}} \, \frac{|\varphi(t,x)|}{P(r)(x)}\; < \; \infty.$$
Moreover, the solution $\varphi$ satisfies the bound
\begin{align*}
    \operatorname{sup}_{0\le t\le T} \lVert \varphi(t,\cdot) \rVert _{\mathscr H^{2+\alpha}_{P(r)}} \; \le \; C\big(\lVert g\rVert _{\mathscr H^{2+\alpha}_{P(r)}}+ \operatorname{sup}_{0\le t\le T} \lVert f(t,\cdot) \rVert _{\mathscr H^{\alpha}_{P(r)}} \big).
\end{align*}
where the constant $C>0$ does not depend on $g, f$ and $\varphi$. The same constant $C$ applies to backward problems posed on any subinterval $[a,b]\subseteq[0,T]$, with terminal time prescribed at $b$.
\end{theorem}

\begin{remark} \label{rem:equivalent_norm_polynomial_growth}
As emphasised in \cite[Lem.~2.2]{palczewski}, for any $k\ge 0$, the norm
$\lVert h/P(r)\rVert _{\mathscr H^{k+\alpha}}$ is equivalent to
$$\sum_{j=0}^{k}
\Big\lVert \frac{\partial_x^j h}{P(r)}\Big\rVert _\infty
+\Big[\frac{\partial_{x}^kh}{P(r)}\Big]_\alpha.$$
Thus, in the context of Theorem~\ref{thm:propagation_of_polynomial_control}, it is precisely the polynomial growth of order $r$ of $g,f$ that propagates to the derivatives of the solution $\varphi$.
\end{remark}

\begin{cor} \label{cor:solution_polynomial_growth_time_derivative_bound}
    Under the assumptions of Theorem \ref{thm:propagation_of_polynomial_control}, the solution $\varphi$ to \eqref{eq:backward_PDE_standard_global_Schauder} also satisfies
    \begin{align*}
    \operatorname{sup}_{\substack{0\le t\le T\\ x \in \R}} \Big| \frac{\partial_t \varphi(t,x)}{P(r+2)(x)}\Big| \; \le \; C \Big( \lVert g \rVert_{\mathscr{H}^{2+\alpha}_{P(r)}} \, + \; \operatorname{sup}_{0\le t \le T} \ \lVert f(t,\cdot) \rVert_{\mathscr{H}^{\alpha}_{P(r)}}\Big)
\end{align*}
where, again, the constant $C>0$ does not depend on $g, f, \varphi$, and the same inequality with the same constant $C$ applies to the backward problem on $[a,b]$.
\end{cor}
\begin{proof}
    Based on the form of the solution \eqref{eq:backward_PDE_standard_global_Schauder}, this is a straightforward consequence of Theorem \ref{thm:propagation_of_polynomial_control} and \eqref{eq:linear_growth}.
\end{proof}

Recall the McKean-Vlasov generator $\mathcal{V}_{t,\nu}$ defined in \eqref{eq:definition_of_operator_mathcal_V} and set 
\begin{align}
    \mathcal{A}_{t,\mu_t}\varphi (x) := D_\nu \langle \mathcal{V}_{t,\nu}\varphi, \nu\rangle(x)\vert_{\nu=\mu_t}   =  \mathcal{V}_{t,\mu_t}\varphi(x)  +  \mathcal{K}_{t,\mu_t}\varphi(x) \label{eq:defi_operator_A}
\end{align} for any $\varphi \in \mathscr{C}_{P(p)}^{0,2}$, where $p\ge 2$ was fixed before \eqref{eq:moments_of_mu} and where 
\begin{align}
    \mathcal{K}_{t,\mu_t}\varphi(x) \; := \; \int_{\R}  \Big[D_\nu (b(y,\mu_t))(x)\,\partial_y\varphi(t,y)
\, + \,
2^{-1} D_\nu(a(y,\mu_t)^2)(x)\,\partial_{yy}\varphi(t,y) \Big]\; \mu_t(\diff y).
\label{eq:definition_of_operator_mathcal_K}
\end{align}
Indeed, if we assume Condition \ref{ass:R}, then the first and second linear functional derivative of the functional $\nu \mapsto \langle \mathcal{V}_{t,\nu}\varphi, \nu\rangle$ exist at the segment between $\mu_t$ and $\mu_t^N$, due to Remark \ref{rem:remark_to_condition_R1}.

Let us now consider the backward PDE associated with the operator $\mathcal{A}_{s,\mu_s}$
\begin{align} \partial_t \varphi(t,x) \; + \; \mathcal{A}_{t,\mu_t}\varphi(x) \; - \; f(t,x) \; = \; 0, \qquad \varphi(T,x)=0, \label{eq:backward_PDE_with_A_operator_Schauder}
\end{align}
which can equivalently be written as 
\begin{align} \partial_t \varphi(t,x) \; + \; \mathcal{V}_{t,\mu_t}\varphi(x) \; - \; \Big[ f(t,x)\; - \;\mathcal{K}_{t,\mu_t}\varphi(x) \Big] \; = \; 0, \qquad \varphi(T,x)=0.\label{eq:backward_PDE_with_A_operator_rewritten_Schauder}
\end{align}

\begin{Condition}{R2} \label{condition:coefficient_regularity_extended_backward_PDE}
There exists $\varepsilon \in (0,1-\alpha)$ such that
$$ \mathcal{K}^o : \; \widetilde{\mathscr{H}}^{1,2+\alpha}_{P(p)} \to \mathscr{H}^{(\alpha+\varepsilon)/2,\alpha+\varepsilon}_{P(p-1)}, \; \; \mathcal K^o\varphi(t,x):=\mathcal{K}_{t,\mu_t}\varphi(x) $$
is a well-defined bounded linear operator.

\end{Condition}

Condition \ref{condition:coefficient_regularity_extended_backward_PDE} ensures that we can establish uniqueness and existence of a suitable solution to \eqref{eq:backward_PDE_with_A_operator_rewritten_Schauder} through a fixed point argument. Although abstract in formulation, Condition \ref{condition:coefficient_regularity_extended_backward_PDE} is satisfied in a wide range of examples and should \emph{by no means} be considered restrictive. Example \ref{ex:affine_kernel_structure} below provides a first indication of this.

\begin{lem} \label{lem:condition_R2_restriction_to_time_interval}
Assume Condition~\ref{condition:coefficient_regularity_extended_backward_PDE}. For every $I=[a,b]\subseteq[0,T]$, $0\le a < b\le T$, there exists a unique bounded linear operator
\begin{align*}
    \mathcal K^o_I: \; \widetilde{\mathscr H}^{1,2+\alpha}_{P(p)}(I\times\mathbb R)\, \to
\, \mathscr H^{(\alpha+\varepsilon)/2,\alpha+\varepsilon}_{P(p-1)}
(I\times\mathbb R)
\end{align*}
such that $\mathcal K_I^o(\restr{v}{I})=\restr{(\mathcal K^o v)}{I}$\,  for all $v \in  \widetilde{\mathscr H}^{1,2+\alpha}_{P(p)}([0,T]\times\mathbb R)$. Moreover, its operator norm is bounded independently of $I$.
\end{lem}


    

We can finally state the existence and uniqueness theorem for solutions of the perturbed parabolic backward PDE \eqref{eq:backward_PDE_with_A_operator_rewritten_Schauder}.

\begin{theorem} \label{thm:propagation_of_polynomial_control_extended_PDE_A}
Assume Condition~\ref{condition:coefficient_regularity_extended_backward_PDE} and Assumptions \ref{ass:Main}--\ref{condition:coefficient_regularity}.
For every interval $I=[a,b]\subseteq[0,T]$, $0\le a<b \le T$, and $f\in\mathscr H^{\alpha/2,\alpha}_{P(p)}(I\times\mathbb R)$, there exists a unique solution $\varphi\in\widetilde{\mathscr H}^{1,2+\alpha}_{P(p)}(I\times\mathbb R)$ to 
$$\partial_t\varphi+\mathcal V_{t,\mu_t}\varphi
+\mathcal K_{t,\mu_t}\varphi-f=0,
\qquad \varphi(b,\cdot)=0, \qquad a\le t \le b.$$
Moreover,
\begin{align}
\lVert \varphi\rVert _{\widetilde{\mathscr H}^{1,2+\alpha}_{P(p)}(I\times\mathbb R)}
\; \le  \; C\lVert f\rVert _{\mathscr H^{\alpha/2,\alpha}_{P(p)}(I\times\mathbb R)},
\label{eq:extended_PDE_uniform_interval_bound}
\end{align}
where $C$ does not depend on $f,\varphi$ and $I$.
\end{theorem}

\begin{remark} \label{rem:extension_for_solutions_on_time_subintervals}
    The solution $\varphi$ on $I \times \R$ can be shown to admit an extension
$\widetilde\varphi\in\widetilde{\mathscr{H}}^{1,2+\alpha}_{P(p)}$ (see, e.g., the proof of Lemma \ref{lem:condition_R2_restriction_to_time_interval}) satisfying
$$\widetilde\varphi|_{I\times\mathbb R}=\varphi \qquad \text{ and }
\qquad \lVert\widetilde\varphi\rVert_{\widetilde{\mathscr{H}}^{1,2+\alpha}_{P(p)}} \; \lesssim  \; \lVert\varphi\rVert_{\widetilde{\mathscr H}^{1,2+\alpha}_{P(p)}(I\times\mathbb R)}.
$$
Whenever a global admissible function is required, we use such
an extension. The extension need not satisfy the PDE outside $I \times \R$. An analogous extension also exists for the solutions of Theorem \ref{thm:propagation_of_polynomial_control} on time subintervals.
\end{remark}


\section{Main results}
\label{sec:MainStat}

This section presents the main results of the article, which establish functional CLTs for fluctuations of the form \eqref{eq:fluctuation_functional_target}.

\subsection{Preliminaries}
\label{subsec:assumptions}
We fix $\alpha \in (0,1)$.  For a finite signed measure $\eta$ with $\int |x|^p \diff|\eta|< \infty$, we define the dual norm
\begin{align}
    \norm{\eta}_{\ast}\; := \;  \operatorname{sup}\Big\{
|\langle g,\eta\rangle| : \,  g \in \mathscr{H}_{P(p)}^{2+\alpha}, \; \lVert g \rVert_{\mathscr{H}_{P(p)}^{2+\alpha}}\le 1. \Big\}\label{eq:defi_ast_norm_and_admissible_class}
\end{align} 
In order to ease notation in the sequel, we also set
\begin{align}
   \mathscr{A} \, := \, \widetilde{\mathscr{H}}^{1,2+\alpha}_{P(p)}.\label{defi_general_admissible_class}
\end{align}
As a first step, we provide a criterion for controlling the second-order terms arising from the Taylor expansion in the measure variable. Such remainders are of the form
$$
\big\langle
\Phi(\varphi,t,\mu_t^{\lambda,N}),
(\mu_t^N-\mu_t)^{\otimes 2}
\big\rangle,
\qquad
\mu_t^{\lambda,N}:=\mu_t+\lambda(\mu_t^N-\mu_t),
$$
where $\Phi$ denotes the corresponding second linear functional derivative. We require these terms to be bounded by the sum of a negligible deterministic contribution and a multiple of the squared dual norm of the empirical measure fluctuation. This criterion will play a central role in our subsequent CLT analysis: it is sufficiently weak to be verified in most examples satisfying the fundamental assumptions introduced in the previous sections, while being strong enough to ensure that the second-order Taylor remainders are asymptotically negligible.

Before stating the assumption, for $K>1$ and $N\ge 1$ define 
\begin{align} \tau^K_N \; := \; \operatorname{inf} \Big\{ t>0 \; : \; \operatorname{sup}_{0\le s \le t} \mathcal{M}_{2p+2}(\mu^N_s) > K\Big\}. \label{defi:localising_stopping_times_for_empirical_moments}
\end{align}
and note that
\begin{align} \lim_{K\to \infty} \,  \operatorname{sup}_{N\ge 1} \, \Pro \big( \tau^K_N \le  T \big) \; = \; 0. \label{eq:stopping_time_vanishing_probabilities}
\end{align}
Indeed, this follows immediately from Markov's inequality together with the uniform moment bound \eqref{eq:moments_of_mu}. Consequently, it suffices to control the stopped remainder processes.

\begin{defi}\label{defi:CLT_fluctuation_compatibility}
    We say that an operator $\Phi:\mathscr{A} \times [0,T] \times \mathcal{P}_p(\R) \to (\R^2 \mapsto \R)$ is \emph{CLT fluctuation compatible}, if for every $K>0$ and $R>0$ there exist 
    $c_N:=c_N(K,R),C_N:=C_N(K,R)>0$ with 
    $N^{1/2}c_N \to 0$, and $N^{-1/2} C_N \to 0$,
    such that 
        \begin{align}
      \limsup_{N\to\infty}\;\operatorname{sup}_{\norm{\varphi}_{\mathscr{A}}\le R} \; \operatorname{sup}_{\, \substack{0\le t\le T \\[1pt] 0\le \lambda \le 1}}\,
\frac{\E\Big[ \big|\langle \Phi(\varphi,t,\mu_t^{\lambda,N}) \, ,\,(\mu_t^N-\mu_t)^{\otimes 2}\rangle\big| \ind_{\{t\le \tau^K_N\}} \Big]}
{c_N + C_N \E\Big[\lVert \mu_t^N-\mu_t\rVert_\ast^2 \ind_{\{t\le \tau^K_N\}} \Big]}
\ \le\ 1. \label{eq:centre_bound_CLT_fluctuation_compatibility}
    \end{align}
\end{defi}

In addition to the regularity Conditions \ref{ass:R} and \ref{condition:coefficient_regularity_extended_backward_PDE}, we require the following assumptions to establish a fluctuation CLT for the quantities \eqref{eq:fluctuation_functional_target}. The aforementioned regularity conditions provide the second-order linearisation of the coefficients in the measure variable (Condition \ref{ass:R}) and allow us to transfer polynomial control to the solutions of the extended backwards PDE \eqref{eq:backward_PDE_with_A_operator_Schauder} (Condition \ref{condition:coefficient_regularity_extended_backward_PDE}). More specifically, the assumptions below ensure that the second-order Taylor remainders arising from the linearisation of the integrated fluctuation functionals in the measure variable are asymptotically negligible.

Sufficient error rates in expectation for the empirical measure fluctuations, quantified in the norm \eqref{eq:defi_ast_norm_and_admissible_class} (and hence for the denominator in \eqref{eq:centre_bound_CLT_fluctuation_compatibility}), will be ensured by Assumption \ref{ass:FB}. Recall the definition of $\mathcal{V}_{t,\nu}$ from \eqref{eq:definition_of_operator_mathcal_V}. 

\begin{Assumption}{A1} \label{ass:FB}
     For all $R>0$ and $K>0$ there exists $L:=L(R,K)>0$, such that 
\begin{align*} \E \Big[ \big| &\langle \mathcal{V}_{t,\mu_t^N}\varphi- \mathcal{V}_{t,\mu_t}\varphi ,\mu_t^N\rangle \big|^2 \ind_{\{t\le \tau^K_N\}} \Big] \; \le \; L \E\Big[\lVert \mu^N_t-\mu_t \rVert_{\ast}^2 \ind_{\{t\le \tau^K_N\}} \Big]  .
\end{align*}
     for all $0\le t\le T$, $N\ge 1$, and $\varphi \in \mathscr{A}$ with $\norm{\varphi}_{\mathscr{A}}\le R$.
\end{Assumption}

We emphasise that Assumption \ref{ass:FB} is primarily a condition on the coefficients $b,a$, in view of the uniform control of the moments of $\mu^N$ up to order $p$ and the restriction $\varphi \in \mathscr{A}$.

Assumption \ref{ass:R2} will then allow us to represent the evaluation of suitable functions against the empirical measure fluctuation through an auxiliary martingale-type decomposition with a negligible remainder.

\begin{Assumption}{A2} \label{ass:R2}
    The second-order linear derivative operator $\Phi: (\varphi, t,\nu) \mapsto D_{\nu\nu}\langle \mathcal{V}_{t,\nu}\varphi, \nu\rangle$ is CLT fluctuation compatible.
\end{Assumption}

We verify Conditions~\ref{ass:R} and
\ref{condition:coefficient_regularity_extended_backward_PDE},
together with Assumptions~\ref{ass:FB}--\ref{ass:R2} for the affine drift Example \ref{ex:affine_kernel_structure}.
For simplicity, we consider a constant diffusion coefficient.
Our results also cover more general drift and diffusion coefficients
under the stated assumptions. The detailed verification is deferred to
Section~\ref{subsec:verification_example}.

\begin{example}[Affine structure] \label{ex:affine_kernel_structure}
Let $\varepsilon>0$, let $\phi:\R \to \R$ be twice continuously 
differentiable with bounded first and second derivatives, and let 
$V \in \widetilde{\mathscr{H}}^{2+\alpha}_{P(p-1)}$, $p\ge 2$, 
such that $V^{1/p}$ is globally Lipschitz and  $V\ge \varepsilon$. Consider the coefficients
\[
b(x,\nu)=\beta x+\gamma\, \phi(\langle V, \nu\rangle^{1/p}),
\qquad
a(x,\nu)=\sigma,
\]
for some fixed constants $\beta, \gamma \in \R$, $\sigma>0$, and $\varepsilon >0$. 
For example, one admissible choice could be $V(y)=\varepsilon + y^{2m}$, with $0\le 2m\le p-1$.
\end{example}

\subsection{Fluctuation bound, martingale-type decomposition and the central limit theorem} \label{subsec:fluctuation_bound}

 We begin by establishing a quantitative bound on the empirical measure fluctuations in the dual norm \eqref{eq:defi_ast_norm_and_admissible_class}. This estimate provides the pivotal control needed to show that the second-order Taylor remainders are asymptotically negligible at the central limit scale.

\begin{prop}[Fluctuation bound] \label{lem:fluctuation_bound}
    Under Assumption \ref{ass:FB} it holds that for every $K>0$ there is a constant $C_K>0$ such that 
    $$ \operatorname{sup}_{0\le t \le T} \E\Big[ \lVert \mu^N_t - \mu_t \rVert_\ast^2 \ind_{\{t\le \tau^K_N\}} \Big] \; \le \; \frac{C_K}{N}.$$
\end{prop} 

\begin{cor}[Bounding the remainder] \label{cor:bounding_the_remainder}
    Let $\Phi: \mathscr{A} \times [0,T] \times \mathcal{P}_p(\R) \to \R$ be twice differentiable at all $\mu_t+\lambda(\mu^N_t-\mu_t)$, $0\le t\le T$, $0\le \lambda \le 1$, in the sense of Definition \ref{defi:functional_derivative}, let its second-order linear functional derivative be CLT fluctuation compatible as described in Definition \ref{defi:CLT_fluctuation_compatibility}, and suppose Assumption \ref{ass:FB} to hold. Then, we have
    \begin{align}
        \Phi(\varphi, t, \mu^N_t) - \Phi(\varphi,t,\mu_t) \; = \; \int_{\R} D_\nu \Phi(\varphi,t,\mu_t)(y) \, (\mu^N_t - \mu_t)(\diff y)  \, + \, R(\varphi,t,\mu^N_t, \mu_t) \label{eq:Taylor_expansion_measure}
    \end{align}
    for all $N\ge 1$, $0\le t\le T$ and $\varphi \in \mathscr{A}$, such that for each $K>0$ and $R>0$,
    $$ \sqrt{N} \operatorname{sup}_{0\le t\le T} \, \operatorname{sup}_{\substack{\, \norm{\varphi}_{\mathscr{A}}\le R}}  \, \E\Big[|R(\varphi,t,\mu^N_t, \mu_t) | \ind_{\{t\le \tau^K_N\}} \Big] \; \conv{N\to \infty} \; \; 0.$$
\end{cor}
\begin{proof}
The corollary follows immediately from \eqref{eq:second_order_Lions_expansion}, Definition \ref{defi:CLT_fluctuation_compatibility}, and Proposition \ref{lem:fluctuation_bound}.
\end{proof}

We next record a martingale-type decomposition result for the fluctuation of the integrated functionals \eqref{eq:fluctuation_functional_target} in Proposition \ref{prop:empirical_measure_decomposition}. While the general decomposition result in Proposition \ref{prop:empirical_measure_decomposition} distinctly extends the existing literature and requires a novel analytical base, the outcome is similar in nature to recent martingale approaches to the fluctuation analysis of interacting particle systems (cf., for instance, \cite{JourdainTse2021} and \cite{CarmonaDelarue2018}).

\begin{prop}[Decomposition of integrated fluctuation] \label{prop:empirical_measure_decomposition}
    Let $\phi \in \mathscr{H}^{\alpha/2,\alpha}_{P(p)}$
    and assume the regularity Conditions \ref{ass:R}, \ref{condition:coefficient_regularity_extended_backward_PDE} and the Assumptions \ref{ass:FB}, \ref{ass:R2}. Then, for every $0< r \le T$ and $N\ge 1$ there is a martingale $M^{N,\phi,r}$ of the form
    $$ M^{N,\phi,r}_t  := \frac{1}{N} \sum_{i=1}^N \int_0^{t\wedge r} a(X^{i,N}_s, \mu^N_s) \; \partial_x \varphi^{\phi,r}(s,X^{i,N}_s) \diff W^i_s$$
    where $\varphi=\varphi^{\phi,r} \in \mathscr{A}$ is a solution to the backward PDE
    \begin{align}
        \partial_t \varphi(t,x) \, + \, \mathcal{A}_{t,\mu_t} \varphi(t,x) \, + \, \phi(t,x) \; = \; 0, \quad \varphi(r,x)= 0, \label{eq:backwards_PDE}
    \end{align}
    for all $t \in[0,r]$, $x \in \R$, and with $\mathcal{A}$ defined in \eqref{eq:defi_operator_A}, such that
    \begin{align}
        \int_0^r \langle \phi(s,\cdot), \sqrt{N}(\mu^N_s - \mu_s) \rangle \diff s \; = \;  
        \sqrt{N} \langle \varphi^{\phi,r}(0,\cdot), \mu^N_0-\mu_0\rangle \; + \; \sqrt{N} M^{N,\phi,r}_r  + \sqrt{N} \int_0^r R_s^{N,\phi,r} \diff s \label{eq:decomposition_additive_fluctuation}
        \end{align}
    for all $N\ge 1$. For each $K>0$, the remainder satisfies
    \begin{align} \sqrt{N} \operatorname{sup}_{0\le t\le r\le T}  \, \E\Big[|R_t^{N, \phi,r} | \ind_{\{t\le \tau^K_N\}} \Big] \; \conv{N\to \infty} \; \; 0. \label{eq:final_decomposition_limit_behaviour_remainder}
    \end{align}
    \end{prop}

\begin{lem}[Quadratic covariation of the martingale parts] \label{lem:quadratic_covariation_M_phi}
 Let $0\le u, r \le T$ and $\phi, \tilde\phi \in \mathscr{H}^{\alpha/2,\alpha}_{P(p)}$. The limiting behaviour of the predictable quadratic covariation of $\sqrt{N} M^{N,\phi,u}$ and $\sqrt{N} M^{N,\tilde \phi,r}$ is given by
    $$ \Big\langle \sqrt{N} M^{N,\phi,u}, \sqrt{N}M^{N, \tilde \phi,r}\Big\rangle_t \, \conv{N\to \infty}{\Pro} \, \E\Big[ \int_0^{t\wedge u\wedge r} a(X_s, \mu_s)^2 \; \partial_x \varphi^{\phi,u}(s,X_s)\; \partial_x \varphi^{\tilde \phi,r}(s,X_s) \diff s \Big] $$
    and we denote the limit by $\Lambda^{\phi,\tilde \phi}_{u,r}(t)$, for each $0\le t\le T.$
\end{lem}

 We are now in a position to state the main result of this paper, which establishes a functional CLT for the integrated fluctuations of the form \eqref{eq:fluctuation_functional_target}.

\begin{theorem}[Functional CLT for integrated fluctuations] \label{thm:functional_CLT_additive_fluctuations}
    Let $\phi \in \mathscr{H}^{\alpha/2,\alpha}_{P(p)}\cap \mathscr{A}$.
    Suppose the regularity Conditions \ref{ass:R}, \ref{condition:coefficient_regularity_extended_backward_PDE} and the Assumptions \ref{ass:FB}, \ref{ass:R2} are satisfied. Then, there exists a continuous centred Gaussian process $Z$ such that
    $$ \Psi^N(\phi) \; \conv{N \to \infty}{\mathcal{L}} \; Z^\phi$$
    on $(\mathscr{C}([0,T];\R), \, \norm{\cdot}_\infty)$,
    where $$\Psi^N(\phi)_t \; := \;  \int_0^t \langle \phi(s,\cdot), \sqrt{N}(\mu^N_s - \mu_s) \rangle \diff s, \qquad 0\le t \le T,$$
    and which has covariance function $\operatorname{Cov}(Z^\phi_t,Z^\phi_r)= \operatorname{Cov}_{\mu_0}\big(\varphi^{\phi,t}(0,\cdot), \varphi^{\phi,r}(0,\cdot) \big) + \Lambda^{\phi,\phi}_{t,r}(t)$, $0\le t\le r\le T$.
\end{theorem}

\begin{remark} \label{rem:remark_to_main_theorem_for_fdd_convergence}
If, in Theorem \ref{thm:functional_CLT_additive_fluctuations}, we only assume
$\phi\in\mathscr{H}^{\alpha/2,\alpha}_{P(p)}$, then the convergence of finite-dimensional distributions continues to hold. The additional assumption $\phi\in\mathscr{A}$ is only used to establish tightness of the sequence $(\Psi^N(\phi))_{N\geq 1}$ in $\mathscr{C}([0,T];\R)$, and hence to obtain functional weak convergence. Moreover, if we only require a fixed-time CLT instead of the functional CLT in Theorem \ref{thm:functional_CLT_additive_fluctuations}, then it suffices to assume a weaker form of the CLT fluctuation compatibility in Assumption \ref{ass:R2}, namely that we do not need uniformity in \eqref{eq:centre_bound_CLT_fluctuation_compatibility} over all $\varphi \in \mathscr{A}$ with $\norm{\varphi}_\mathscr{A} \le R$ inside the limit superior in $N$, but instead it is enough for \eqref{eq:centre_bound_CLT_fluctuation_compatibility} to hold for each $\varphi \in \mathscr{A}$ with $\norm{\varphi}_\mathscr{A} \le R$ separately.
\end{remark}

\section{Application: statistical inference}
\label{sec:Stat}
In this section, we apply the weak limit theory developed in Section~\ref{sec:MainStat} to two statistical testing problems for interacting particle systems. The first concerns the temporal constancy of macroscopic population characteristics that may enter the mean-field interaction, while the second concerns parametric hypothesis testing for the drift.

We assume that the complete particle trajectories
$$\big(X_t^{i,N}\big)_{0\le t\le T},
\qquad i=1,\ldots,N,$$
are observed continuously over the time interval $[0,T]$. This continuous-time observation scheme should be viewed as the idealised high-frequency limit of a discretely observed experiment.

\subsection{Testing temporal constancy of population characteristics}
In many McKean--Vlasov models, the dependence of the coefficients on the population law is mediated through a finite collection of observables. Let
$$\phi_1,\ldots,\phi_q \, \in \, \mathscr{H}^{\alpha/2,\alpha}_{P(p)}\cap \mathscr{A}$$
be known functions and define
$$
{\bf m}_t := {\bf m}(\mu_t)
:=
\big(
\langle \phi_1,\mu_t\rangle,\ldots,
\langle \phi_q,\mu_t\rangle
\big)^\top,
\qquad t\in[0,T].
$$
The process $({\bf m}_t)_{t\in[0,T]}$ describes a collection of population characteristics. Of particular interest are often polynomial observables, for instance $\phi_j(x)=x^{k_j}$, $k_j\le p$, since moments of the population law naturally arise in many interaction models. Our objective is to test whether these population characteristics remain constant over the observed time horizon. More precisely, we consider
\begin{align}
\mathcal H_0:
\,
{\bf m}_t={\bf m}_0
\; \, \text{for all }t\in[0,T]
\qquad
\text{against}
\qquad
\mathcal H_1:
\; 
{\bf m}_t\neq{\bf m}_0
\; \, \text{for some }t\in[0,T].
\label{test:constancy}
\end{align}
This hypothesis has a particularly natural interpretation when the dependence of the drift on the population law is of the form $b(x,\mu)
=
B(x,{\bf m}(\mu))$.
In this case, the quantities on which ${\bf m}_t$ acts are not merely descriptive features of the population, but they determine the mean-field feedback appearing in the response of each particle. Under $\mathcal H_0$, this feedback remains unchanged along the McKean--Vlasov flow, which in this specific case implies
$$
b(x,\mu_t)
=
B\bigl(x,{\bf m}_0\bigr),
\qquad
t\in[0,T].
$$
To characterise $\mathcal H_0$, define the \(\R^q\)-valued process
\begin{align}
{\bf F}_t
:=
\int_0^T \Big(\frac{t}{T}-\ind_{\{s\le t\}}\Big)\, {\bf m}_s\diff s,
\qquad
t\in[0,T].
\label{eq:def_F_constancy}
\end{align}
Then it holds that
$$
\mathcal H_0
\quad\Longleftrightarrow\quad
{\bf F}_t=0
\ \text{for all }t\in[0,T]
\quad\Longleftrightarrow\quad
\sup_{t\in[0,T]}\norm{{\bf F}_t}=0.
$$
Indeed, constancy of \({\bf m}\) immediately implies \({\bf F}\equiv0\). Conversely, if \({\bf F}_t=0\) for every \(t\), then
$$
\int_0^t{\bf m}_s\diff s
\; = \; 
\frac{t}{T}\int_0^T{\bf m}_s\diff s,
\qquad
t\in[0,T].
$$
Due to Assumption \ref{ass:Main}, ${\bf m}$ is continuous, and hence ${\bf m}_t = T^{-1}\int_0^T{\bf m}_s\diff s$.

The process ${\bf F}$ admits an immediate empirical counterpart, by replacing $\mu$ by $\mu^N$, which we denote $\widehat{\bf F}^N$. A direct application of our main result Theorem \ref{thm:functional_CLT_additive_fluctuations} now characterises the limit fluctuation of $\widehat{{\bf F}}^N$ around ${\bf F}$.

\begin{theorem} \label{ThFirstTest}
    Under the assumptions of Theorem \ref{thm:functional_CLT_additive_fluctuations}, there is weak functional convergence 
    $$ \sqrt{N}\big(\widehat{{\bf F}}^N - {\bf F}\big) \; \conv{N \to \infty}{\mathcal{L}} \; {\bf Z}$$
    on $(\mathscr{C}([0,T];\R^q), \, \norm{\cdot}_\infty)$, where ${\bf Z}$ is a $q$-dimensional centred Gaussian process with covariance function given by 
    $$ \operatorname{Cov}({\bf Z}^{(i)}_t,{\bf Z}^{(j)}_r) \; = \; \tilde \Lambda^{\phi_i,\phi_j}_{t,r}(t) \, - \, \frac{t}{T}\tilde \Lambda^{\phi_i,\phi_j}_{T,r}(r) \, - \, \frac{r}{T}\tilde \Lambda^{\phi_i,\phi_j}_{t,T}(t) \, + \, \frac{tr}{T^2}\tilde \Lambda^{\phi_i,\phi_j}_{T,T}(T), \qquad 0\le t\le r\le T, $$
    $i,j \in \{1,...,q\}$, with
    $$ \tilde \Lambda^{\phi_i,\phi_j}_{t,r}(t) := \; \operatorname{Cov}_{\mu_0}\big(\varphi^{\phi_i,t}(0,\cdot), \varphi^{\phi_j,r}(0,\cdot) \big) \, + \, \Lambda^{\phi_i,\phi_j}_{t,r}(t),$$
    where $\Lambda$ is defined as in Lemma \ref{lem:quadratic_covariation_M_phi} and  $\varphi^{\phi,r}$ is the solution to the backward PDE \eqref{eq:backwards_PDE}. Moreover, by continuous mapping, under $\mathcal{H}_0$ it holds that $\sqrt{N} \sup_{0\le t\le T}\lVert{\widehat{\bf F}^N_t \rVert}$ converges in law to $\sup_{0\le t\le T}\norm{{\bf Z}_t}$. 
\end{theorem}


\subsection{Parametric testing for the drift} \label{subsec:parametric_testing}
We next consider a finite-dimensional parametric model for the drift. Let
\(b_1,\ldots,b_d:\R\times\mathcal P_p(\R)\to\R\)
be known functions satisfying the same regularity, growth, and moment assumptions as \(b\), as specified in Section~\ref{sec:setup}, and define
\begin{align*}
V
:=
\Bigg\{
b_\theta
=
\sum_{\ell=1}^d\theta_\ell b_\ell
:
\theta=(\theta_1,\ldots,\theta_d)^\top\in\R^d
\Bigg\}.
\end{align*}
Suppose that \(b\in V\), so that \(b=b_{\theta^*}\) for some unknown parameter \(\theta^*\in\R^d\). Given a matrix
\(\Xi\in\R^{r\times d}\) of rank \(r\le d\) and a vector \(\xi\in\R^r\), we wish to test
\begin{align}
\mathcal H_0:
\;
\Xi\theta^*=\xi
\qquad
\text{against}
\qquad
\mathcal H_1:
\; 
\Xi\theta^*\neq\xi.
\label{test:projection}
\end{align}
This testing problem is related to parametric inference for the drift component of interacting particle systems, which has been recently studied in e.g. \cite{AmorinoHeidariPilipauskaitePodolskij2023,SharrockKantasParpasPavliotis2023}.

For measurable functions
\(f,g:\R\times\mathcal P_p(\R)\to\R\), define the bilinear form
\[
\Psi(f,g)
:=
\int_0^T
\big\langle
f(\cdot,\mu_t)g(\cdot,\mu_t),
\mu_t
\big\rangle
\diff t.
\]
Assume that \(b_1,\ldots,b_d\) are linearly independent in
\[
L^2\big([0,T]\times\R,\mu_t(\diff x)\diff t\big).
\]
Then the Gram matrix
\[
{\bf \Pi}
:=
\big[
\Psi(b_\ell,b_k)
\big]_{\ell,k=1,\ldots,d}
\]
is positive definite and therefore invertible. Defining
\[
{\bf \Gamma}
:=
\big[
\Psi(b_\ell,b)
\big]_{\ell=1,\ldots,d},
\]
the identity \(b=b_{\theta^*}\) and bilinearity of $\Psi$ yield
\[
{\bf \Gamma}
=
{\bf \Pi}\theta^*,
\]
and consequently
\[
\theta^*
=
{\bf \Pi}^{-1}{\bf \Gamma}.
\]
Thus, asymptotically valid inference for the hypothesis \eqref{test:projection} reduces to studying empirical counterparts of \(\Psi(b_\ell,b_k)\) and \(\Psi(b_\ell,b)\). The entries of the Gram matrix ${\bf \Pi}$ are estimated by
\begin{align}
\widehat{\Psi}(b_j,b_k)_N
:=
\frac{1}{N}
\sum_{i=1}^N
\int_0^T
b_j\big(X_t^{i,N},\mu_t^N\big)
b_k\big(X_t^{i,N},\mu_t^N\big)
\diff t,
\qquad
j,k=1,\ldots,d.
\label{eq:estimator_Psi_bj_bk}
\end{align}
The mixed terms involving the unknown drift are estimated via the semimartingale dynamics,
\begin{align}
\widehat{\Psi}(b_j,b)_N
:=
\frac{1}{N}
\sum_{i=1}^N
\int_0^T
b_j\big(X_t^{i,N},\mu_t^N\big)
\diff X_t^{i,N},
\qquad
j=1,\ldots,d.
\label{eq:estimator_Psi_bj_b}
\end{align}

To derive the asymptotic behaviour of the estimators, we linearise the corresponding functionals
$$ \nu\mapsto \langle f(\cdot,\nu)g(\cdot,\nu),\nu\rangle,
\qquad f,g\in \{b_1,\ldots,b_d,b\},$$ 
around the McKean--Vlasov law $\mu_t$. In order to apply Proposition \ref{prop:empirical_measure_decomposition} to the resulting first-order terms and to control the associated second-order Taylor remainders, we impose the following additional assumption.

\begin{Assumption}{A3} \label{ass:R3}
    For all $f,g \in \{ b_1,...,b_d,b\}$ the following hold:
    \begin{enumerate}[(i)]
        \item the map $\phi_{(f,\, g)}:(x,t)\mapsto  D_\nu \langle f(\cdot,\mu_t)g(\cdot,\mu_t),\mu_t\rangle(x)$ belongs to $\mathscr{H}^{\alpha/2,\alpha}_{P(p)}$;
    \item the second-order linear derivative operator $\Phi: (\varphi, t,\nu) \mapsto D_{\nu\nu}\langle f(\cdot,\nu)g(\cdot,\nu),\nu\rangle$ is CLT fluctuation compatible.
    \end{enumerate}
\end{Assumption}

\noindent
For the formulation of the CLT, we introduce, for all $f,g,\tilde f, \tilde g \in \{b_1,...,b_d,b\}$,
$$ \widehat{Z}_N^{(f,g)} := \widehat{\Psi}(f,g)_N- \Psi(f,g) \qquad \text{ and } \qquad  C_0\big((f,g),(\tilde f, \tilde g)\big) :=\operatorname{Cov}_{\mu_0}\big(\varphi^{\phi_{(f,g)},T}(0,\cdot),\varphi^{\phi_{(\tilde f,\tilde g)},T}(0,\cdot) \big).$$
Recall the definition of $\Lambda$ in Lemma \ref{lem:quadratic_covariation_M_phi}. We will now state the CLT for the relevant quantities of this section.

\begin{theorem}[CLT for estimators]  \label{thm:CLT_for_estimators}
    Assume the regularity Conditions \ref{ass:R}, \ref{condition:coefficient_regularity_extended_backward_PDE} and the Assumptions \ref{ass:FB}--\ref{ass:R3}. Then, the following multivariate CLT holds:
    \begin{align*}
            \sqrt{N} \widehat{\bf{Z}}_N  :=  \sqrt{N} \begin{bmatrix*}[l]
            \operatorname{vec}\big(\widehat{Z}_N^{(b_j,b)}\big)_{j=1,...,d} \\
            \operatorname{vec}\big(\widehat{Z}_N^{(b_j,b_k)}\big)_{j,k=1,...,d} 
        \end{bmatrix*} 
        \; \conv{N\to \infty}{ \mathcal{L}} \;  \begin{bmatrix*}[l]
            \operatorname{vec}\big(Z^{(b_j,b)}\big)_{j=1,...,d}\\
            \operatorname{vec}\big(Z^{(b_j,b_k)}\big)_{j,k=1,...,d} 
        \end{bmatrix*}  =: {\bf{Z}} \; \sim \; \Norm_{d+d^2}({\bf{0}}, {\bf{\Sigma}}),
    \end{align*}
    where 
    \begin{align*}
    \begin{cases}
        \operatorname{Cov}(Z^{(b_j,b)}, Z^{(b_k,b)}) \;&=\; \Psi\big(a\cdot (b_j+\partial_x \varphi^{\phi_{(b_j,b)},T}), \, a\cdot \big(b_k+\partial_x \varphi^{\phi_{(b_k,b)},T}\big) \big) \, + \,  C_0\big((b_j,b),(b_k,b)\big) \\
        \operatorname{Cov}(Z^{(b_j,b_k)}, Z^{(b_\ell,b)}) \;&= \;\Psi\big(a\cdot \partial_x \varphi^{\phi_{(b_j,b_k)},T}, \,a\cdot \big(b_\ell+\partial_x \varphi^{\phi_{(b_\ell,b)},T}\big)  \big) \, + \,  C_0\big((b_j,b_k),(b_\ell,b)\big)\\
        \operatorname{Cov}(Z^{(b_j,b_k)}, Z^{(b_\ell,b_m)}) \; &= \;\Psi\big(a\cdot \partial_x \varphi^{\phi_{(b_j,b_k)},T}, \, a \cdot \partial_x \varphi^{\phi_{(b_\ell,b_m)},T}  \big) \, + \,  C_0\big((b_j,b_k),(b_\ell,b_m)\big).
        \end{cases}
    \end{align*}
\end{theorem}
Denote by $\widehat{{\bf \Gamma}}_N$ and $\widehat{{\bf \Pi}}_N$ the respective estimators of ${\bf \Gamma}$ and ${\bf \Pi}$, and let $\widehat{{\bf \Pi}}_N^{-1}$ denote the pseudoinverse of $\widehat{{\bf \Pi}}_N$. 

\begin{theorem}[CLT for the projections] \label{thm:CLT_for_projection_test}
    Under the notation of \eqref{test:projection}, it holds that 
    $$ \sqrt{N} \big(\Xi\widehat{{\bf \Pi}}_N^{-1} \widehat{{\bf \Gamma}}_N - \Xi {\bf \Pi}^{-1} {\bf \Gamma}\big) \; \conv{N\to \infty}{ \mathcal{L}} \; \Norm_r({\bf 0}, J_G  {\bf \Sigma} \, J_G^\top),$$
    where $J_G:= J_G ({\bf \Gamma}, {\bf \Pi})$ is the Jacobian matrix at $({\bf \Gamma}, {\bf \Pi})$ of the map $G: \R^d \times \operatorname{GL}_d(\R) \to \R^r$, $G({\bf y}, {\bf E}):= \Xi \, {\bf E}^{-1}{\bf y}$, with respect to the coordinates $({\bf y}, \operatorname{vec}({\bf E}))$.
\end{theorem}

\begin{remark} (Variance estimation)
To formally construct the testing procedures discussed in this section, we need consistent estimators of the asymptotic variances appearing in Theorems \ref{ThFirstTest} and \ref{thm:CLT_for_estimators}. In this remark, we focus on the setting of Theorem \ref{thm:CLT_for_estimators}, which constitutes the more involved case. Our aim is to explain the main idea behind the proposed estimator rather than to provide a rigorous proof of its consistency.

To estimate the asymptotic covariance matrix $\Sigma$, we propose a subsampling procedure, as a direct estimation of this quantity appears to be rather difficult. Let $K=K_N$ denote the number of blocks and partition the set of particle indices $\{1,\ldots,N\}$ into disjoint blocks
\[
B_1,\ldots,B_K,
\]
each containing $N/K$ particles. We assume that
\[
K\to\infty
\qquad\text{and}\qquad
\frac{N}{K}\to\infty
\]
as $N\to\infty$. For each block $B_k$, $k=1,\ldots,K$, define
\begin{align}
\widehat{Z}_{N,k}^{(f,g)}
:=
\widehat{\Psi}(f,g)_{N,k}
-
\widehat{\Psi}(f,g)_{N},
\label{defhatv}
\end{align}
where $\widehat{\Psi}(f,g)_{N,k}$ is defined analogously to
$\widehat{\Psi}(f,g)_{N}$, but is computed using only the particles
$X^i$ with indices $i\in B_k$. In particular, all empirical measures
entering $\widehat{\Psi}(f,g)_{N,k}$ are constructed exclusively from
the observations corresponding to particles in $B_k$. The vector
$\widehat{\mathbf Z}_{N,k}$ is defined analogously.

We then define the subsampling estimator of the asymptotic covariance
matrix $\Sigma$ by
\[
\widehat{\Sigma}
:=
\frac{1}{K}
\sum_{k=1}^K
\frac{N}{K}
\widehat{\mathbf Z}_{N,k}
\widehat{\mathbf Z}_{N,k}^{\top}.
\]
The intuition behind this estimator is as follows. Since each block
contains $N/K$ particles, the normalization $\sqrt{N/K}$ corresponds to the
scaling in the central limit theorem applied at the block level.
Moreover, since the blocks are disjoint and the particles become
asymptotically independent in the propagation-of-chaos regime, the
random matrices
\[
\frac{N}{K}
\widehat{\mathbf Z}_{N,k}
\widehat{\mathbf Z}_{N,k}^{\top},
\qquad k=1,\ldots,K,
\]
can be viewed as asymptotically independent proxies for the covariance
matrix $\Sigma$. Averaging over an increasing number of blocks therefore
suggests, by a weak-law-of-large-numbers argument, that
\(\widehat{\Sigma}\xrightarrow{\mathbb P}\Sigma.\)
Consistency of an analogous subsampling procedure has been established
rigorously in \cite{heidari_podolskij} for i.i.d.\ observations from a
McKean--Vlasov model. We expect the same principle to apply in the
present interacting-particle setting, but omit a formal proof here.
\end{remark}

\section{Proofs} \label{sec:Proofs}

\subsection{Proofs to Section \ref{subsec:technicalities}}
We begin with the proof of Theorem \ref{thm:propagation_of_polynomial_control}.

\begin{proof}[Proof of Theorem \ref{thm:propagation_of_polynomial_control}]
We prove the assertion on an arbitrary time interval
$I=[a,b]\subseteq[0,T]$, with $a<b$ and terminal condition $g$ at time $b$. In order to do so, we will transform the problem into an equivalent problem to which we can apply \cite[Thm.~3.8]{lorenzi}.

Write $P=P(r)$ and consider $\varphi=Pv$ for any $v \in \mathscr{C}^{1,2}(I\times \R)$. Since $P$ does not depend on time,
\[
\partial_t\varphi=Pv_t,\qquad
\partial_x\varphi=P'v+P \partial_x v,\qquad
\partial_{xx}\varphi=P''v+2P' \partial_x v+P\partial_{xx} v.
\]
Substituting these identities into the backward PDE \eqref{eq:backward_PDE_standard_global_Schauder} restricted on the time interval $I$ instead of $[0,T]$, and dividing
by $P$ gives the equivalent problem 
\begin{align}
\partial_t v+q \partial_{xx}v +L \partial_x v +\ell v- \tilde f=0,
\qquad v(b,\cdot)=\tilde g, \label{eq:proof_first_Schauder_equivalent_PDE}
\end{align}
where $q:=2^{-1}\sigma^2$ and 
\[
h_1:=P'P^{-1},\; \; 
h_2:=P''P^{-1},\; \; 
L:=\beta+2qh_1,\; \; 
\ell :=qh_2+\beta h_1, \; \; 
\tilde g :=g/P,\; \; 
\tilde f:=f/P.
\]
The assumptions on $g,f$ imply
$\tilde g \in\mathscr H^{2+\alpha}$ as well as $\operatorname{sup}_{0\le t \le T}\lVert \tilde f(t,\cdot)\rVert_{\mathscr{H}^\alpha}< \infty$. Moreover, for every $j\ge 0$,
\[
|\partial_x^jh_1(x)|
\lesssim (1+|x|)^{-1-j},
\qquad
|\partial_x^jh_2(x)|
\lesssim (1+|x|)^{-2-j}.
\]
The linear growth \eqref{eq:linear_growth} implies
\[
|\beta(t,x)|+|\sigma(t,x)|\lesssim 1+|x|,
\qquad
\lambda/2\le q (t,x)\lesssim 1+x^2.
\]
Differentiating $q$ and using Assumption \ref{condition:coefficient_regularity}(i) gives, for
$1\le k\le3$,
\[
|\partial_x^k q(t,x)|\lesssim 1+|x|,\qquad |\partial_x^kq(t,x)|
\lesssim 1+|\sigma(t,x)|\lesssim q(t,x).
\]
where we have estimated $q$ by $|\sigma|\ge\sqrt\lambda$ in the last inequality.
The product rule therefore yields
\[
|L(t,x)|\lesssim 1+|x|,\qquad
\sup_{\substack{0\le t\le T\\x \in \R}}|\partial_x^kL(t,x)|<\infty, \qquad \sup_{\substack{0\le t\le T\\x \in \R}}|\partial_x^m \ell (t,x)|<\infty,
\]
for $1\le k\le 3$ and $0\le m\le 3$. Note that the omitted constants do not depend on $f,g$ and $v$. By Assumption \ref{condition:coefficient_regularity}(ii), the coefficients $q,L,\ell$ and their spatial derivatives
up until order three are jointly continuous and have uniformly bounded (in time) local spatial $\mathscr H^\alpha$ norms.

To put the equation in suitable forward form, write $d=b-a$, reverse time
by setting $\tau=b-t$, and put
$$
w(\tau,x):=v(b-\tau,x), \quad \hat q(\tau,x):=q(b-\tau,x), \quad \hat L(\tau,x):=L(b-\tau,x),$$
$$\hat \ell(\tau,x):=\ell(b-\tau,x), \quad\hat f(\tau, x):= \tilde f(b-\tau, x)$$
for $0\le\tau\le d$.
The equation \eqref{eq:proof_first_Schauder_equivalent_PDE} for $v$ is then equivalent to the forward-in-time equation
\begin{align*}
\partial_\tau w -\hat q \partial_{xx}w -\hat L \partial_x w -\hat \ell w + \hat f=0,
\qquad w(0,\cdot)=\tilde g.\label{eq:proof_second_Schauder_second_equivalent_PDE}
\end{align*}
In order to ensure the existence of a solution to this forward problem, we will invoke \cite[Thm.~3.8]{lorenzi}. To this end, we verify the hypotheses of the theorem, which---above all---amounts to the verification of \cite[Hypotheses~2.1]{lorenzi}, with $\nu=q$. Note that it suffices to consider $g, L, \ell$ since $\hat g, \hat L, \hat \ell$ are simply time-reversed versions of the former. From the bounds above, we deduce ellipticity (2.1(ii)), the diffusion coefficient derivative conditions (2.1(v)), and the potential bound (2.1.(iv)). The growth conditions (2.1(iii)) hold trivially due to
\[
|x|q\le(1+x^2)q,\qquad
q\le(1+x^2)q,\qquad
xL\, \lesssim \, (1+x^2) \, \lesssim \, (1+x^2)q.
\]
The drift and potential derivative bounds (2.1(vi)) allow the functions
$d,r_0,\varrho$ in \cite[Hypotheses~2.1(vi)]{lorenzi}
to be chosen as positive constants. Since $q\ge\lambda/2$,
their required inequality holds after increasing its constant.
Finally, for $H(x)=1+x^2$,
\[
2q(t,x)+2xL(t,x)+ \ell(t,x) (1+x^2) \, \lesssim \, 
H(x),
\]
again as a result of the bounds above, which verifies the
Lyapunov condition (2.1(vii)). The local spatial regularity (modified 2.1(i)) required by \cite[Thm.~3.8]{lorenzi} is given due to Assumption \ref{condition:coefficient_regularity}(ii) and the fact that $h_1,h_2 \in \mathscr{C}^\infty(\R)$. Finally, the required regularity of $\hat f, \tilde g$ follows directly from $g \in \mathscr{H}^{2+\alpha}_{P(r)}$ and $f \in \mathscr{C}([0,T]\times \R)$ with $\operatorname{sup}_{0\le t\le T}\lVert f(t,\cdot)\rVert_{\mathscr{H}^\alpha_{P(r)}}< \infty.$ 

To obtain a constant independent
of $I$, extend the transformed problem from $[0,d]$ to the fixed
interval $[0,T]$ by defining
\[
\eta_I(\tau):=b-\operatorname{min}\{\tau,d\},
\qquad
F_I(\tau,x):=\hat{f}(\eta_I(\tau),x).
\]
Now, the coefficients of the PDE
\begin{align}
\partial_\tau w - \hat q(\eta_I(\cdot),\cdot) \partial_{xx}w - \hat L(\eta_I(\cdot),\cdot) \partial_x w - \hat \ell(\eta_I(\cdot),\cdot) w + F_I =0,
\qquad w(0,\cdot)=\tilde g\label{eq:proof_third_Schauder_second_equivalent_PDE}
\end{align}
are reversed on
$[0,d]$ and frozen at $a$ afterwards. All the verified regularity
and structural bounds are preserved, independently of $I$. Therefore \cite[Thm.~3.8]{lorenzi} guarantees the existence of a unique solution $w_I$ to \eqref{eq:proof_third_Schauder_second_equivalent_PDE} satisfying
$\operatorname{sup}_{\tau,x} |w_I(\tau,x)| < \infty$. Moreover, it holds that 
\begin{align*}
    \operatorname{sup}_{0\le \tau \le T} \lVert w_I(\tau,\cdot) \rVert _{\mathscr H^{2+\alpha}} \; &\le \; C\big(\lVert g\rVert _{\mathscr H^{2+\alpha}}+ \operatorname{sup}_{0\le \tau\le T} \lVert F_I(\tau ,\cdot) \rVert _{\mathscr H^{\alpha}} \big)\\
    &\le \; C\big(\lVert g\rVert _{\mathscr H^{2+\alpha}}+ \operatorname{sup}_{t \in I} \lVert f(t ,\cdot) \rVert _{\mathscr H^{\alpha}_{P(r)}} \big)
\end{align*}
for some constant $C>0$. According to \cite[Rem.~2.8(i)]{lorenzi}, the fixed time horizon and the common structural bounds ensure
that $C>0$ is independent of $I,f,g$. Reverting $w_I$ back to $v=\varphi/P$ yields
\[
\sup_{t\in I}
\lVert\varphi(t,\cdot)\rVert_{\mathscr H^{2+\alpha}_{P(r)}}
\le C\big(
\lVert g\rVert_{\mathscr H^{2+\alpha}_{P(r)}}
+\sup_{t\in I}
\lVert f(t,\cdot)\rVert_{\mathscr H^\alpha_{P(r)}}
\big).
\]

Finally, the substitutions $\varphi=Pv$ and
$w(\tau,x)=v(b-\tau,x)$ are invertible and map
every classical solution in the prescribed growth class to a bounded
classical solution of the transformed forward problem on $[0,d]$.
Uniqueness therefore follows from \cite[Thm.~3.8]{lorenzi}
on that interval.
\end{proof}

\begin{proof}[Proof of Lemma \ref{lem:condition_R2_restriction_to_time_interval}]
Let $h:=b-a$ and set 
$$ \gamma_h(s):= \; \frac{hs}{h+2s}, \qquad s\ge 0.$$
Define the extension operator $E_I: \widetilde{\mathscr H}^{1,2+\alpha}_{P(p)}(I\times\mathbb R) \to \widetilde{\mathscr H}^{1,2+\alpha}_{P(p)}([0,T]\times\mathbb R)$ by
$$ (E_Iu)(t,x) := \; \begin{cases} 
2u(a,x) - u(a+\gamma_h(a-t),x), \quad &\text{if } t<a,\\
u(t,x), \quad &\text{if } a\le t \le b,\\
2u(b,x) - u(b-\gamma_h(b-t),x), \quad &\text{if } b<t.\\
\end{cases}$$
Note that 
$$0\le \gamma_h \le h/2, \quad \gamma_h(0)=0, \quad 0< \gamma_h'\le 1, \quad \text{and } \; \; \gamma_h'(0)=1.$$
Therefore, indeed $E_Iu \in \widetilde{\mathscr H}^{1,2+\alpha}_{P(p)}([0,T]\times\mathbb R)$ and for any $u \in \widetilde{\mathscr H}^{1,2+\alpha}_{P(p)}(I\times\mathbb R)$ it holds that 
$$ \lVert E_Iu\rVert_{\widetilde{\mathscr H}^{1,2+\alpha}_{P(p)}([0,T]\times\mathbb R)} \; \le \; 3\lVert u\rVert_{\widetilde{\mathscr H}^{1,2+\alpha}_{P(p)}(I\times\mathbb R)}.$$
Define $\mathcal K^o_Iu:=\restr{(\mathcal K^o(E_Iu))}{I}$ and note that, since the value of $K^ov$ at time $t$ only depends on $v(t,\cdot)$ and $E_I(\restr{v}{I})(t,\cdot)=v(t,\cdot)$ for all $t\in I$, 
$$ (\mathcal{K}^o(E_I(v|_I)))(t,x) \; = \; (\mathcal K^o v)(t,x), \qquad t\in I.$$
Finally, we deduce
\begin{align*}
    \lVert \mathcal K^o_Iu\rVert_{\widetilde{\mathscr H}^{1,2+\alpha}_{P(p)}(I\times\mathbb R)} \; \le \; \lVert \mathcal K^o(E_Iu)\rVert_{\widetilde{\mathscr H}^{1,2+\alpha}_{P(p)}([0,T] \times\mathbb R)} &\le \norm{\mathcal K^o}_{\text{op}} \lVert E_Iu\rVert_{\widetilde{\mathscr H}^{1,2+\alpha}_{P(p)}([0,T]\times\mathbb R)} \\
    &\le \; 3 \norm{\mathcal K^o}_{\text{op}} \lVert u\rVert_{\widetilde{\mathscr H}^{1,2+\alpha}_{P(p)}(I\times\mathbb R)}.
\end{align*}
\end{proof}

Now, we will begin with the following uniqueness lemma for solutions to the extended backward PDE \eqref{eq:backward_PDE_with_A_operator_Schauder}.
\begin{lem}[Uniqueness lemma] \label{lem:uniqueness_lemma_PDE_proof}
    Suppose that Condition \ref{condition:coefficient_regularity_extended_backward_PDE} holds, and let $\lambda \in [0,1]$ and $I=[c,d]\subseteq [0,T]$, $0\le c<d\le T$. 
    Then, any solution $\varphi \in \widetilde{\mathscr{H}}^{1,2+\alpha}_{P(p)}(I \times \R)$ to the backward PDE
    \begin{align} \partial_t \varphi(t,x) \, + \, \mathcal{V}_{t,\mu_t}\varphi(x) \, + \, \lambda(\mathcal{K}_{I}^o \varphi)(t,x) \; = \; 0, \qquad \varphi(d,x)=0, \label{eq:extended_backward_PDE_for_uniqueness_lemma} \end{align}
    satisfies $\varphi(t,x)= 0$ for all $c\le t\le d$, $x \in \R$.
\end{lem}
\begin{proof}
Denote by $L>0$ the operator norm bound from Condition \ref{condition:coefficient_regularity_extended_backward_PDE}.

In a first step, let $h \in \mathscr{H}^{2+\alpha}_{P(p)}$ and denote by $\bar{h}(s,x):=h(x)$ its constant-in-time extension. We are going to show that for every $0\le t\le T$,
\begin{align}
        \lVert \mathcal{K}^o\bar h(t,\cdot) \rVert_{\mathscr{C}^{0}_{P(p-1)}} \; &\le \; L \norm{h}_{\mathscr{H}^{2+\alpha}_{P(p)}} \label{eq:uniqueness_proof_1}\\
    \lVert \mathcal{K}^o\bar h(t,\cdot) \rVert_{\mathscr{H}^{\alpha}_{P(p-1)}} \; &\le \; C_TL \norm{h}_{\mathscr{H}^{2+\alpha}_{P(p)}}^{1-\varepsilon/2} \norm{h}_{\mathscr{C}^0_{P(p)}}^{\varepsilon/2}  \label{eq:uniqueness_proof_2}
\end{align}
for some constant $C_T>0$ only dependent on $T$. The estimate  \eqref{eq:uniqueness_proof_1} follows immediately from Condition \ref{condition:coefficient_regularity_extended_backward_PDE}. Now, choose a smooth $\chi \in \mathscr{C}^\infty_c((-1,1))$ with $\chi(0)=1$ and set 
$$ v_{\delta}(s,x)\; := \; \chi\Big( \frac{s-t}{\delta}\Big) h(x), \qquad 0< \delta \le T/2.$$
Clearly, $v_{\delta} \in \widetilde{\mathscr{H}}^{1,2+\alpha}_{P(p)}$ and $v_{\delta}(t,\cdot)=h$, and therefore, by Condition \ref{condition:coefficient_regularity_extended_backward_PDE},
\begin{align*}
    \frac{|\mathcal{K}^o \bar{h}(t,x)|}{P(p-1)(x)\delta^{(\alpha+\varepsilon)/2}} \; &= \; \frac{|\mathcal{K}^o v_{\delta}(t,x)-\mathcal{K}^o v_{\delta}(s_\delta,x)|}{P(p-1)(x)(|t-s_\delta|^{1/2} + |x-x|)^{(\alpha+\varepsilon)/2}}  \\
    &\le \; \big[\mathcal{K}^o v_{\delta}\big]_{\alpha+\varepsilon} \le \; L \lVert v_{\delta} \rVert_{\widetilde{\mathscr{H}}^{1,2+\alpha}_{P(p)}} \; \le \; C_\chi L \big(\norm{h}_{\mathscr{H}^{2+\alpha}_{P(p)}}+ \delta^{-1}\norm{h}_{\mathscr{C}^{0}_{P(p)}}\big) 
\end{align*} 
where $s_\delta \in [0,T]$ is chosen such that $|t-s_\delta|=\delta$ and $C_\chi$ only depends on the choice of $\chi$. Hence, choosing 
$$\delta:=\operatorname{min}\{1,T/2\} \, \norm{h}_{\mathscr{C}^0_{P(p)}} \norm{h}_{\mathscr{H}^{2+\alpha}_{P(p)}}^{-1} \; \le \; 1$$
we obtain 
\begin{align}
     \lVert \mathcal{K}^o \bar{h}(t,\cdot) \rVert_{\mathscr{C}^0_{P(p-1)}} \; &\le \; C_\chi L \delta^{(\alpha+\varepsilon)/2}  \big(\norm{h}_{\mathscr{H}^{2+\alpha}_{P(p)}}+ \delta^{-1}\norm{h}_{\mathscr{C}^{0}_{P(p)}}\big) \notag \\
    &\le \; C_{T,\chi} L \norm{h}_{\mathscr{H}^{2+\alpha}_{P(p)}}^{1-(\alpha+\varepsilon)/2} \norm{h}_{\mathscr{C}^{0}_{P(p)}}^{(\alpha+\varepsilon)/2}
    \label{eq:uniqueness_proof_3},
\end{align}
where $C_{T,\chi}>0$ only depends on $T,\chi$.
For any $\theta:=(\alpha+\varepsilon)$-Hölder continuous function $z$ with Hölder constant $H:=\norm{z}_{\mathscr{H}^\theta}$, it holds that $|z(x)-z(y)|\le \operatorname{min}\{2\norm{z}_\infty, H|x-y|^\theta\}\le (2\norm{z}_\infty)^{1-\alpha/\theta}H^{\alpha/\theta}|x-y|^\alpha$ due to the classical Hölder inequality, and hence it follows that $[z]_\alpha\le 2^{1-\alpha/\theta}\norm{z}_\infty^{1-\alpha/\theta} H^{\alpha/\theta}$. In addition, since $\norm{z}_\infty\le H$, we deduce $\norm{z}_{\mathscr{H}^\alpha}\le 4\norm{z}_\infty^{1-\alpha/\theta} \norm{z}_{\mathscr{H}^\theta}^{\alpha/\theta}$. Consequently, we obtain from \eqref{eq:uniqueness_proof_3} (together with Remark \ref{rem:equivalent_norm_polynomial_growth}) and Condition \ref{condition:coefficient_regularity_extended_backward_PDE},
\begin{align}
    \lVert \mathcal{K}^o \bar{h}(t,\cdot) \rVert_{\mathscr{H}^\alpha_{P(p-1)}} \; &\le \; 4\lVert \mathcal{K}^o \bar{h}(t,\cdot) \rVert_{\mathscr{C}^0_{P(p-1)}}^{1-\alpha/\theta} \lVert \mathcal{K}^o \bar{h}(t,\cdot) \rVert_{\mathscr{H}^\theta_{P(p-1)}}^{\alpha/\theta} \notag \\
    &\le \; 4 C_{T,\chi}^{1-\alpha/\theta} L^{1-\alpha/\theta} \norm{h}_{\mathscr{H}^{2+\alpha}_{P(p)}}^{(1-\theta/2)(1-\alpha/\theta)} \norm{h}_{\mathscr{C}^{0}_{P(p)}}^{\theta/2(1-\alpha/\theta)} L^{\alpha/\theta} \norm{h}_{\mathscr{H}^{2+\alpha}_{P(p)}}^{\alpha/\theta} \notag \\
    &\le \; 4 C_{T,\chi}^{1-\alpha/\theta} L \norm{h}_{\mathscr{H}^{2+\alpha}_{P(p)}}^{1-\varepsilon/2} \norm{h}_{\mathscr{C}^{0}_{P(p)}}^{\varepsilon/2}.  \notag
\end{align}
By multiplication by $P(p-1)/P(p)$ we deduce, by virtue of Remark \ref{rem:equivalent_norm_polynomial_growth} and since $P(p-1)/P(p)\le 1$,
\begin{align}
    \lVert \mathcal{K}^o \bar{h}(t,\cdot) \rVert_{\mathscr{H}^\alpha_{P(p)}} \; \le \; 4 C_{T,\chi}^{1-\alpha/\theta} L \norm{h}_{\mathscr{H}^{2+\alpha}_{P(p)}}^{1-\varepsilon/2} \norm{h}_{\mathscr{C}^{0}_{P(p)}}^{\varepsilon/2}.  \label{eq:uniqueness_proof_4}
\end{align}

We can now proceed to showing uniqueness of the solutions. We will do so consecutively on short intervals. Denote by $v$ the extension $E_I\varphi$ of the solution $\varphi$ to \eqref{eq:extended_backward_PDE_for_uniqueness_lemma}, as introduced in the proof of Lemma \ref{lem:condition_R2_restriction_to_time_interval}. In particular, recall that $v=\varphi$ on $I\times \R$. By Lemma \ref{lem:condition_R2_restriction_to_time_interval}, $f=f_I:=-\lambda \mathcal{K}^o_I\varphi=-\lambda (\mathcal{K}^o_I v)|_I$ has the regularity required by Theorem \ref{thm:propagation_of_polynomial_control}. The restricted-time-interval estimate of the latter therefore gives
\begin{align*}\operatorname{sup}_{s \in I} \norm{\varphi(s,\cdot)}_{\mathscr{H}^{2+\alpha}_{P(p)}} \; \le \; C\operatorname{sup}_{s \in I} \norm{f_I(s,\cdot)}_{\mathscr{H}^{\alpha}_{P(p)}} \; &= \; \lambda C\operatorname{sup}_{s \in I} \norm{\mathcal{K}^o_I v(s,\cdot)}_{\mathscr{H}^{\alpha}_{P(p)}} \\ &= \; \lambda C \operatorname{sup}_{s \in I} \norm{\mathcal{K}^o v(s,\cdot)}_{\mathscr{H}^{\alpha}_{P(p)}}.
\end{align*}
Applying \eqref{eq:uniqueness_proof_4} yields
\begin{align}
     \operatorname{sup}_{s \in I} \norm{\varphi(s,\cdot)}_{\mathscr{H}^{2+\alpha}_{P(p)}} \; \le \; C\operatorname{sup}_{s \in I} \norm{f_I(s,\cdot)}_{\mathscr{H}^{\alpha}_{P(p)}} \; \le \; C_T L \operatorname{sup}_{s \in I} \norm{\varphi(s,\cdot)}_{\mathscr{H}^{2+\alpha}_{P(p)}}^{1-\varepsilon/2} \norm{\varphi(s,\cdot)}_{\mathscr{C}^{0}_{P(p)}}^{\varepsilon/2}\label{eq:uniqueness_proof_5}
\end{align}
for some constant $C_T>0$ only depending on $T$. Now, for each $t \in I$ and $x \in \R$, let $X_s^{t,x}$ be the solution of
$$
\diff X_s^{t,x}= b(X_s^{t,x},\mu_s)\diff s + a(X_s^{t,x},\mu_s)\diff W_s, \qquad X_t^{t,x}=x, \; t\le s \le T.$$ Indeed, here we understand $\mu_s$, $0\le s \le T$, as fixed and the regularity of the coefficients yields a strong solution to this ordinary time-inhomogeneous SDE. For each $R>0$, define the stopping time
$$
\tau^x_R := \operatorname{inf}\big\{ s\ge t: |X_s^{t,x}|\ge R \big\} \wedge d.
$$
Applying Itô's formula to the process $s\mapsto \varphi(s,X_s^{t,x})$
yields in particular
\begin{align*}
\varphi\bigl(\tau_R^x,X_{\tau^x_R}^{t,x}\bigr)-\varphi(t,x) =
\int_t^{\tau^x_R}
\Big[
\partial_s \varphi
+
\mathcal V_{s,\mu_s}\varphi
\Big]
\bigl(s,X_s^{t,x}\bigr)
\diff s \,
+
\int_t^{\tau^x_R}
a(X_s^{t,x},\mu_s)
\partial_x \varphi\bigl(s,X_s^{t,x}\bigr)
\diff W_s.
\end{align*}
By exploiting \eqref{eq:extended_backward_PDE_for_uniqueness_lemma}, taking expectation, and letting $R\to \infty$, we obtain
\begin{align}
\varphi(t,x)=
\lambda \mathbb E\Big[ \int_t^d \mathcal K^o v \bigl(s,X_s^{t,x}\bigr)\diff s\Big],
\label{eq:form_of_solution_to_uniqueness_lemma_backward_PDE}
\end{align}
where we have used Fubini-Tonelli's theorem and dominated convergence, which is applicable due to the moment condition \eqref{eq:moments_of_mu} and the fact that $\mathcal K^o v$ has at most polynomial growth of order $p-1$ in the spatial variable uniformly in time. Using Remark \ref{rem:refined_explanation_of_moment_propagation} and, again, \eqref{eq:uniqueness_proof_1}, gives 
\begin{align*} |\varphi(t,x)| \; &\le \; \lambda \operatorname{sup}_{s \in I} \lVert K^o v(s,\cdot) \rVert_{\mathscr{C}^0_{P(p-1)}} \int_t^d \E[P(p-1)(X^{t,x}_s)] \diff s \\
&\le \; C_T \lambda L (d-t) P(p-1)(x) \operatorname{sup}_{s \in I} \norm{\varphi(s,\cdot)}_{\mathscr{H}^{2+\alpha}_{P(p)}} 
\end{align*}
for some constant $C_T>0$ only depending on $T$ and $p$, which implies
$$ \operatorname{sup}_{s \in I} \norm{\varphi(s,\cdot)}_{\mathscr{C}^0_{P(p)}} \; \le \; C_T |I| L  \operatorname{sup}_{s \in I} \norm{\varphi(s,\cdot)}_{\mathscr{H}^{2+\alpha}_{P(p)}}. $$
Substituting this into \eqref{eq:uniqueness_proof_5} yields
\begin{align} \operatorname{sup}_{s \in I} \norm{\varphi(s,\cdot)}_{\mathscr{H}^{2+\alpha}_{P(p)}} \; \le \; \tilde C_{T} L^{1+\varepsilon/2} |I|^{\varepsilon/2} \operatorname{sup}_{s \in I} \norm{\varphi(s,\cdot)}_{\mathscr{H}^{2+\alpha}_{P(p)}} \label{eq:uniqueness_proof_6}
\end{align}
for some constant $\tilde C_T$ only depending on $T$ and $p$. 

Finally, choose $\ell_0>0$ such that $\tilde C_T L^{1+\varepsilon/2} \ell_0^{\varepsilon/2}<1$. Then, $\varphi \equiv 0$ on $[c,d]\times \R$ if $d-c\le \ell_0$. If $d-c>\ell_0$, divide the interval $[c,d]$ into pieces of length at most $\ell_0$ and repeat the same argument backwards finitely many times, starting from terminal time $d$, in order to obtain $\varphi \equiv 0$ on $[c,d]\times \R$. Indeed, this is possible since the prefactors on the right-hand side of \eqref{eq:uniqueness_proof_6} do not depend on $I$ (beyond its length) or $\varphi$ itself.
\end{proof}

\begin{proof}[Proof of Theorem \ref{thm:propagation_of_polynomial_control_extended_PDE_A}]
Let 
$$\Gamma_{\mathcal{V}}: \; \mathscr{H}^{\alpha/2,\alpha}_{P(p)}(I \times \R) \to \widetilde{\mathscr{H}}^{1,2+\alpha}_{P(p)}(I\times \R)=: \mathscr{X}_I$$
be the solution map for the parabolic backward PDE \eqref{eq:backward_PDE_standard_global_Schauder} formulated for the time interval $I$ instead of $[0,T]$, with $g\equiv 0$, i.e. it takes $f$ and returns a solution $\varphi$ to the backward PDE. By Theorem \ref{thm:propagation_of_polynomial_control} and Corollary \ref{cor:solution_polynomial_growth_time_derivative_bound}, this map is linear and bounded uniformly in $I$.

Suppose $\varphi \in \mathscr{X}_I$ solves 
\begin{align}
\partial_t \varphi(t,x) \; + \;\mathcal{V}_{t,\mu_t} \varphi (x) \; + \; \lambda(\mathcal{K}_{I}^o \varphi)(t,x) \; = \; f(t,x),
\qquad
\varphi(b,x)=0, \label{eq:proof_existence_backward_PDE_0}
\end{align}
for some $\lambda \in [0,1]$ and all $a\le t\le b$ and $x \in \R$. Let $J:=[c,d]\subseteq I$, $c<d$, be arbitrary. Proceeding in full analogy to the proof of Lemma \ref{lem:uniqueness_lemma_PDE_proof}, we obtain the estimates
\begin{align*}
\operatorname{sup}_{s \in J} \norm{\varphi(s,\cdot)}_{\mathscr{C}^0_{P(p)}} \; &\le \; C_T \Big( \norm{\varphi(d,\cdot)}_{\mathscr{H}^{2+\alpha}_{P(p)}} + |J|  \operatorname{sup}_{s \in I} \norm{f(s,\cdot)}_{\mathscr{H}^{\alpha}_{P(p)}} \\
& \qquad \qquad + |J| L  \operatorname{sup}_{s \in J} \norm{\varphi(s,\cdot)}_{\mathscr{H}^{2+\alpha}_{P(p)}}\Big), \\
    \operatorname{sup}_{s \in J} \norm{\varphi(s,\cdot)}_{\mathscr{H}^{2+\alpha}_{P(p)}} \; &\le \; C_T \Big(\norm{\varphi(d,\cdot)}_{\mathscr{H}^{2+\alpha}_{P(p)}} + \operatorname{sup}_{s \in I} \norm{f(s,\cdot)}_{\mathscr{H}^{\alpha}_{P(p)}}  \\
& \qquad \qquad + L \operatorname{sup}_{s \in J} \norm{\varphi(s,\cdot)}_{\mathscr{H}^{2+\alpha}_{P(p)}}^{1-\varepsilon/2} \norm{\varphi(s,\cdot)}_{\mathscr{C}^0_{P(p)}}^{\varepsilon/2} \Big),
\end{align*}
for some constant $C_T>0$ independent of $I,J, \lambda, f$ and $\varphi$. Applying Young's inequality to the last summand on the right-hand side of the second inequality yields
\begin{align*}
    \operatorname{sup}_{s \in J} \norm{\varphi(s,\cdot)}_{\mathscr{H}^{2+\alpha}_{P(p)}} \; &\le \; C' \Big(\norm{\varphi(d,\cdot)}_{\mathscr{H}^{2+\alpha}_{P(p)}} + \operatorname{sup}_{s \in I} \norm{f(s,\cdot)}_{\mathscr{H}^{\alpha}_{P(p)}} +  \operatorname{sup}_{s \in J} \norm{\varphi(s,\cdot)}_{\mathscr{C}^0_{P(p)}} \Big),
\end{align*}
with $C'>0$ depending only on $C_T, L$ and $\varepsilon$. Substituting the early bound for $\operatorname{sup}_{s \in J} \norm{\varphi(s,\cdot)}_{\mathscr{C}^0_{P(p)}}$ into this inequality gives

\begin{align*}
    \operatorname{sup}_{s \in J} \norm{\varphi(s,\cdot)}_{\mathscr{H}^{2+\alpha}_{P(p)}} \; &\le \; C'' \Big(\norm{\varphi(d,\cdot)}_{\mathscr{H}^{2+\alpha}_{P(p)}} + \operatorname{sup}_{s \in I} \norm{f(s,\cdot)}_{\mathscr{H}^{\alpha}_{P(p)}}\Big) +  C''' |J| \operatorname{sup}_{s \in J} \norm{\varphi(s,\cdot)}_{\mathscr{H}^{2+\alpha}_{P(p)}},
\end{align*}
for constants $C'', C'''>0$ dependending only on $C_T,T,L$ and $\varepsilon$. Choose $\ell_0>0$ such that $C'''\ell_0<1$. Then, whenever $|J|\le \ell_0$, we deduce
$$ \operatorname{sup}_{s \in J} \norm{\varphi(s,\cdot)}_{\mathscr{H}^{2+\alpha}_{P(p)}} \; \le \; \tilde C \Big(\norm{\varphi(d,\cdot)}_{\mathscr{H}^{2+\alpha}_{P(p)}} + \operatorname{sup}_{s \in I} \norm{f(s,\cdot)}_{\mathscr{H}^{\alpha}_{P(p)}}\Big) $$
for some constant $\tilde C>0$ depending only on $C'',C'''$ and $\ell_0$. Partitioning $I$ into finitely many intervals of length at most $\ell_0$ and
iterating backwards from $\varphi(b,\cdot)\equiv 0$ therefore gives
$$\sup_{s\in I}\lVert \varphi (s,\cdot)\rVert_{\mathscr{H}^{2+\alpha}_{P(p)}} \, \le\,  \hat{C} \operatorname{sup}_{s \in I} \norm{f(s,\cdot)}_{\mathscr{H}^{\alpha}_{P(p)}},$$
for some constant $\hat{C}>0$ depending on the same quantities as $\tilde C$. Together with Corollary \ref{cor:solution_polynomial_growth_time_derivative_bound}, up to a potential increase of the constant, we deduce
\begin{align} \sup_{s\in I}\lVert \varphi (s,\cdot)\rVert_{\mathscr{X}_I} \, \le\,  \hat{C} \operatorname{sup}_{s \in I} \norm{f(s,\cdot)}_{\mathscr{H}^{\alpha}_{P(p)}}.  \label{eq:proof_existence_extended_backward_PDE_1}
\end{align}

We will now use Schaefer’s fixed point theorem \cite{schaefer}, which applies to our case as follows. If an operator $\mathcal{S}:\mathscr{X}_I \to \mathscr{X}_I$ is continuous and compact, and the set 
\begin{align}\big\{\varphi \in \mathscr{X}_I \; : \; \varphi =\lambda \mathcal{S}(\varphi),\ \lambda \in[0,1] \big\} \label{eq:schaefer_set} \end{align}
is bounded, then $\mathcal{S}$ has a fixed point. Since 
\begin{align*} \mathscr{H}^{(\alpha+\varepsilon)/2,\alpha+\varepsilon}_{P(p-1)}(I\times \R) \subseteq \mathscr{H}^{\alpha/2,\alpha}_{P(p)}(I\times \R)\text{, the operator }\\
\mathcal{S}:= \Gamma_{\mathcal{V}}\circ (f-\mathcal{K}^o_I): \widetilde{\mathscr{H}}^{1,2+\alpha}_{P(p)}(I\times \R) \to \widetilde{\mathscr{H}}^{1,2+\alpha}_{P(p)}(I\times \R) 
\end{align*}
is well-defined and it is linear. By Lemma \ref{lem:condition_R2_restriction_to_time_interval}, Theorem \ref{thm:propagation_of_polynomial_control} and Corollary \ref{cor:solution_polynomial_growth_time_derivative_bound}, it is bounded, hence continuous. To see that it is also compact, note first that $\Gamma_{\mathcal{V}}$ is linear and bounded, the latter by Theorem \ref{thm:propagation_of_polynomial_control} and Corollary \ref{cor:solution_polynomial_growth_time_derivative_bound}. Thus, it is enough to show that $\mathcal{K}^o_I$ is compact, since the composition of a compact linear operator with a bounded linear, hence continuous, operator is compact. Making use of the Arzelà-Ascoli tightness characterisation on compact spatial intervals together with a Hölder interpolation, a classical procedure shows that the canonical embedding 
$$ \mathscr{H}^{(\alpha+\varepsilon)/2,\alpha+\varepsilon}_{P(p-1)}(I\times \R) \; \hookrightarrow \; \mathscr{H}^{\alpha/2,\alpha}_{P(p)}(I\times \R)$$
is compact, which gives the compactness of $\mathcal{K}^o_I$. Finally, it only remains to show that the set 
 \eqref{eq:schaefer_set} is bounded. This, however, follows immediately from \eqref{eq:proof_existence_extended_backward_PDE_1}. Schaefer's fixed point theorem now yields a solution to \eqref{eq:proof_existence_backward_PDE_0}. Uniqueness is implied by Lemma \ref{lem:uniqueness_lemma_PDE_proof}.
\end{proof}

\subsection{Verification of the conditions for Example \ref{ex:affine_kernel_structure}} \label{subsec:verification_example}
The coefficients $b$ and $a$ satisfy Assumption \ref{ass:Main}. Assumption \ref{condition:coefficient_regularity} is trivial in this affine example due to the uniform moment bounds of \eqref{eq:moments_of_mu} and the fact that $V$ grows at most polynomially with degree $p-1$. Further, the coefficients admit infinitely many linear functional derivatives in their measure variable hence Condition \ref{ass:R} is satisfied. Assumption \ref{ass:FB} holds true since, by definition of $\norm{\cdot}_\ast$ and on $\{t\le \tau^K_N\}$,
\begin{align*}
    \big| \langle \mathcal{V}_{t,\mu_t^N}\varphi- \mathcal{V}_{t,\mu_t}\varphi ,\mu_t^N\rangle \big| \; &= \; |\gamma| \, \int_{\R} | \phi(\langle V, \mu^N_t\rangle^{1/p}) -  \phi(\langle V, \mu_t\rangle^{1/p})| |\partial_x\varphi(t,x)| \, \mu^N_t(\diff x) \\  
    &\lesssim \; \frac{|\gamma|}{p\varepsilon^{1-1/p}}\,  \int_{\R} | \langle V, \mu^N_t -\mu_t \rangle| |\partial_x\varphi(t,x)| \, \mu^N_t(\diff x)\\
    &\lesssim \;  | \langle V, \mu^N_t -\mu_t \rangle| \int_{\R}  |x|^p \, \mu^N_t(\diff x) \; \lesssim \; \lVert \mu^N_t- \mu_t \rVert_*,
\end{align*}
where we have used the Lipschitz continuity of $\phi$, the mean-value theorem for the map $z\mapsto z^{1/p}$ on $[\varepsilon,\infty)$, as well as that $\varphi,V \in \mathscr{A}$.
Moreover, towards Condition \ref{condition:coefficient_regularity_extended_backward_PDE}, we have 
$$D_\nu b(x,\nu)(y)
=
\gamma\,\phi'\big(\langle V,\nu\rangle^{1/p}\big)
\frac{1}{p}\,
\langle V,\nu\rangle^{1/p-1}
V(y).$$ 
Using that $V \in \widetilde{\mathscr{H}}^{2+\alpha}_{P(p-1)}$, $V\ge \varepsilon$, that $\phi'$ is Lipschitz, $\varphi \in \mathscr{A}$ and the uniform moment bounds of \eqref{eq:moments_of_mu}, it is straightforward to see that Condition \ref{condition:coefficient_regularity_extended_backward_PDE} is satisfied.  

Computing the linear functional derivatives of the operators appearing in Assumption \ref{ass:R2}, we obtain
\begin{align*}
D_{\nu\nu}
\big\langle \mathcal{V}_{t,\nu}\varphi,\nu\big\rangle(y,z)
={}&
D_\nu b(y,\nu)(z)\,\partial_x\varphi(t,y)
+
D_\nu b(z,\nu)(y)\,\partial_x\varphi(t,z)
\\
&+
\int_{\mathbb R}
D_{\nu\nu}b(x,\nu)(y,z)\,
\partial_x\varphi(t,x)\,\nu(\diff x),
\end{align*}
again, if $\langle V, \nu\rangle >0$, and it is zero otherwise, where
\begin{align*}
D_{\nu\nu}b(x,\nu)(y,z)
={}&
\frac{\gamma}{p^2}
\Big[
\phi''\big(\langle V,\nu\rangle^{1/p}\big)
\langle V,\nu\rangle^{2/p-2}
\\
&\qquad
+
(1-p)\phi'\big(\langle V,\nu\rangle^{1/p}\big)
\langle V,\nu\rangle^{1/p-2}
\Big]
V(y)V(z).
\end{align*}
With $\eta^N_t:=\mu_t^N-\mu_t$ and $\mu^{\lambda, N}_t:=\mu_t+\lambda(\mu_t^N-\mu_t)$ this implies that
\begin{align}
\E\Big[&\big|\big\langle D_{\nu\nu}\langle \mathcal{V}_{t,\mu^{\lambda, N}_t}\varphi,\mu^{\lambda, N}_t\rangle,\;(\eta_t^N)^{\otimes2}\big\rangle \big| \ind_{\{t\le \tau^K_N\}} \Big] \notag \\ 
&\lesssim \; \E\big[\lVert \eta^N_t\rVert_* \,|\langle \partial_x\varphi(t,\cdot),\eta_t^N\rangle|\ind_{\{t\le \tau^K_N\}} \big] \, + \, \E\big[\lVert \eta^N_t\rVert_*^2\ind_{\{t\le \tau^K_N\}} \big],
\label{ex:specific_form_of_second_derivative_R2}
\end{align}
where we have also used \eqref{eq:moments_of_mu} to obtain the second bound on the right-hand side of \eqref{ex:specific_form_of_second_derivative_R2}. Let $K>1$, $R>0$, and $\varphi \in \mathscr{A}$ with $\lVert \varphi\rVert_{\mathscr{A}}\le R$. If we can bound $\lVert \eta^N_t\rVert_* |\langle \partial_x\varphi(t,\cdot),\eta_t^N\rangle| \le c_n(K,R) + C_N(K,R) \lVert \eta^N_t \rVert_*^2$ on $\{t\le \tau^N_K\}$, with $N^{1/2}c_N(K,R) \to 0$, and $N^{-1/2} C_N(K,R) \to 0$ as $N\to \infty$, then we deduce Assumption \ref{ass:R2}. Clearly, we cannot directly deduce this property since, in general, $\partial_x\varphi$ is not even twice differentiable. However, we do know that $\varphi \in \mathscr{A}$. In the sequel we describe a general approach to obtaining Assumption \ref{ass:R2} in such situations. Set
\begin{align} \gamma_N := N^\varsigma, \, \delta_N:= N^{-\beta}, \qquad \text{ for some } \beta \in (0,1/4) \text{ and } \varsigma \in (0, \beta/p). \label{eq:example_choices_of_gammaN_and_deltaN} \end{align}
Define the cut-off--mollification $\bar{\varphi}^N(t,\cdot):=\bar{\varphi}^{\gamma_N,\delta_N}(t,\cdot)$ as in Lemmas \ref{lem:mollification_lemma_control_of_derivatives}, \ref{lem:mollificiation_lemma_approximation}.
By Lemma \ref{lem:mollification_lemma_control_of_derivatives} applied with $r=p$ and $k=1,...,4$, there exists $C_1>0$ (independent of $t$, $N$, and $\varphi$) such that
\begin{equation*}
\operatorname{sup}_{0\le t\le T} \lVert \partial_x\bar\varphi^{N}(t,\cdot)\rVert_{\mathscr{H}^{2+\alpha}_{P(p)}}
\; \le \; C_1R \big(1+\delta_N^{-2}\big).
\end{equation*}
Consequently, by the definition of $\lVert \cdot\rVert_\ast$,
\begin{equation}
|\langle \partial_x\bar\varphi^{N}(t,\cdot),\eta_t^N\rangle|\; 
\le \; \lVert \partial_x\bar\varphi^{N}(t,\cdot)\rVert_{\mathscr{H}^{2+\alpha}_{P(p)}}\,\lVert \eta_t^N\rVert_\ast
\; \le \; C_1 R\big(1+\delta_N^{-2}\big)\,\lVert \eta_t^N\rVert_\ast. \label{ex:first_part_decomposition_bound}
\end{equation}
Next, Lemma \ref{lem:mollificiation_lemma_approximation} with $r=p$ gives that
\begin{equation}
\operatorname{sup}_{0\le t\le T}\operatorname{sup}_{|x|\le \gamma_N}\big|\partial_x\bar\varphi^{N}(t,x)-\partial_x\varphi(t,x)\big|
\; \le \; C_1R\,\delta_N\,(1+\gamma_N^{p}). \label{ex:approximation_error_bound}
\end{equation}
Decompose
\begin{equation}
\langle \partial_x\varphi(t,\cdot),\eta_t^N\rangle
=\langle \partial_x\bar\varphi^{N}(t,\cdot),\eta_t^N\rangle
+\langle \partial_x\varphi(t,\cdot)-\partial_x\bar\varphi^{N}(t,\cdot),\eta_t^N\rangle. \label{ex:decomposition_of_test_function_part}
\end{equation}
The first term is controlled by \eqref{ex:first_part_decomposition_bound}. For the second term, we use the triangle inequality and split 
\begin{align}
\big|\langle &\partial_x\varphi(t,\cdot)-\partial_x\bar\varphi^{N}(t,\cdot),\eta_t^N\rangle\big| \notag \\
&\le \; \int_{|x|\le \gamma_N}\big|\partial_x\varphi(t,x)-\partial_x\bar\varphi^{N}(t,x)\big|\,|\eta_t^N|(\diff x)
     \, + \, \int_{|x|> \gamma_N}\big|\partial_x\varphi(t,x)-\partial_x\bar\varphi^{N}(t,x)\big|\,|\eta_t^N|(\diff x)\notag\\
&\le \; 2\operatorname{sup}_{|x|\le \gamma_N}\big|\partial_x\varphi(t,x)-\partial_x\bar\varphi^{N}(t,x)\big|
     \, + \, \int_{|x|> \gamma_N}\big(|\partial_x\varphi(t,x)|+|\partial_x\bar\varphi^{N}(t,x)|\big)\,|\eta_t^N|(\diff x). \label{ex:decomposition_inequality_for_approximation_error}
\end{align}
While the first term on the right-hand side of \eqref{ex:decomposition_inequality_for_approximation_error} is controlled by \eqref{ex:approximation_error_bound}, for the second term in \eqref{ex:decomposition_inequality_for_approximation_error}, use the polynomial growth of $\partial_x \varphi(t,\cdot)$ and the polynomial growth bound for $\partial_x \bar{\varphi}^N(t,\cdot)$ (from Lemma \ref{lem:mollification_lemma_control_of_derivatives} with $k=1$) to obtain on $\{t\le \tau^K_N\}$ 
\begin{align}
\int_{|x|> \gamma_N}&\big(|\partial_x\varphi(t,x)|+|\partial_x\bar\varphi^{N}(t,x)|\big)\,|\eta_t^N|(\diff x) \notag \\
&\lesssim \; \gamma_N^{-(p+2)} R \int_{\R}1+|x|^{2p+2}\,|\eta_t^N|(\diff x) \notag \\
&\lesssim \; \gamma_N^{-(p+2)} R \big(\mathcal{M}_{2p+2}(\mu_t^N)+ \mathcal{M}_{2p+2}(\mu_t)+2\big) \; \lesssim \; C_2 \gamma_N^{-(p+2)}, \label{ex:tail_bound}
\end{align}
for some constant $C_2:=C_2(K,R)>0$, where we have used \eqref{eq:moments_of_mu} and the fact that $\mathcal{M}_{2p+2}(\mu_t^N)\le K$ on $\{t\le \tau^K_N\}$. Indeed, on $\{|x|>\gamma_N\}$ one has $1+|x|^p \lesssim \gamma_N^{-(p+2)}(1+|x|^{2p+2})$. Combining \eqref{ex:first_part_decomposition_bound}--\eqref{ex:tail_bound}, we obtain
\begin{align*}
|\langle \partial_x\varphi(t,\cdot),\eta_t^N\rangle| \; 
\lesssim \;
C_3 \big[ (1+\delta_N^{-2})\,\lVert\eta_t^N\rVert_\ast 
      +\delta_N(1+\gamma_N^p)+\gamma_N^{-(p+2)}\big], \qquad N\ge 1, \, 0\le t \le T,
\end{align*}
for some constant $C_3:=C_3(K,R)>0$. By Young's inequality, for any $\lambda>0$ we have
\begin{align*}
\lVert\eta_t^N\rVert_* \big[\delta_N(1+\gamma_N^p)+\gamma_N^{-(p+2)}\big]
 \; \le \; 
\tfrac12N^\lambda\lVert\eta_t^N\rVert_*^2
+\tfrac12N^{-\lambda}
\big[\delta_N(1+\gamma_N^p)+\gamma_N^{-(p+2)}\big]^2.
\end{align*}
Now, choose $\lambda>0$ such that 
$$ \max\!\left\{2\beta,\tfrac12-2(\beta-p\varsigma), \tfrac12-2(p+2)\varsigma\right\}<\lambda<\tfrac12,$$
and set 
\begin{align*}
    c_N(K,R) \; &:= \; C_3 N^{-\lambda}\big[ \delta_N(1+\gamma^p_N) + \gamma^{-(p+2)}_N\big]^2 \\
    C_N(K,R) \; &:= \; C_3 (1+\delta_N^{-2} + N^\lambda).
\end{align*}
The required convergence rates now follow from the choices of $\gamma_N, \delta_N$ and $\lambda$.

\subsection{Proofs to Section \ref{subsec:fluctuation_bound}}
\begin{proof}[Proof of Proposition \ref{lem:fluctuation_bound}]
     Let $g$ be an admissible function for \eqref{eq:defi_ast_norm_and_admissible_class}. By Theorem \ref{thm:propagation_of_polynomial_control} and Corollary \ref{cor:solution_polynomial_growth_time_derivative_bound}, applied on $[0,t]$, the backward PDE
     \begin{align} \partial_s \varphi_t(s,x)\, + \, \mathcal V_{s,\mu_s}\varphi_t(x)\, \; = \;0,
\qquad 
\varphi_t(t,x)\, = \, g(x)  \label{eq:backward_PDE_with_operator_V} \end{align}
    has a solution $\varphi_t$, with extension $\varphi$ to the whole time horizon $[0,T]$ (see Remark \ref{rem:extension_for_solutions_on_time_subintervals}), satisfying
    \begin{align} \norm{\varphi}_{\mathscr{A}} \; \le \; C_1\norm{g}_{\mathscr{H}^{2+\alpha}_{P(p)}} \; \le \; C_1 \label{eq:norm_varphi_bounded_fluctuation_bound_proof}
    \end{align}
    for some constant $C_1>0$ which does not depend on $\varphi, g$ or $t$.
     Applying Itô's formula to each $\varphi(t,X_t^{i,N})$ and then averaging over $i=1,...,N$ yields the decomposition
    $$ \langle g, \mu_t^N \rangle \; = \; \langle \varphi(t,\cdot), \mu_t^N \rangle \; = \; \langle \varphi(0,\cdot),\mu_0^N \rangle
+\int_0^t \langle \partial_s \varphi(s,\cdot) + \mathcal{V}_{s,\mu_s^N}\varphi ,\mu_s^N\rangle\,\diff s\; 
+ \;M^{N,\varphi}_t,$$
where
    $$ M^{N,\varphi}_t \; := \; \frac{1}{N}\sum_{i=1}^N\int_0^t \partial_x \varphi(s, X_s^{i,N})\,  a(X_s^{i,N},\mu_s^N)\,\diff W_s^i $$
is a continuous martingale. Now, from \eqref{eq:backward_PDE_with_operator_V}, $\partial_s \varphi(s,\cdot)=-\mathcal V_{s,\mu_s}\varphi(s,\cdot)(x)$, and so we obtain 
$$  \langle g , \mu_t^N \rangle \; = \; \langle \varphi(0,\cdot),\mu_0^N \rangle
+\int_0^t \langle \mathcal{V}_{s,\mu_s^N}\varphi -\mathcal{V}_{s,\mu_s}\varphi,\mu_s^N\rangle\,\diff s\; 
+ \;M^{N,\varphi}_t .$$
Similarly, $\langle g, \mu_t \rangle = \langle \varphi(0,\cdot),\mu_0 \rangle$ since $\langle \varphi(t,\cdot),\mu_t\rangle= \E[\varphi(t,X_t)]$ and after applying Itô's lemma the expectation of the martingale part equals zero. Subtracting $\langle g , \mu_t \rangle $ from $ \langle g , \mu_t^N \rangle$ then gives us
\begin{align*}
\langle g, \mu_t^N-\mu_t \rangle \; = \; \langle \varphi(0,\cdot),\mu_0^N-\mu_0 \rangle \; + \; \int_0^t \langle \mathcal{V}_{s,\mu_s^N}\varphi -\mathcal{V}_{s,\mu_s}\varphi ,\mu_s^N\rangle \, \diff s  \; + \; M^{N,\varphi}_t.
\end{align*}
Fix $K>0$. After squaring both sides, due to Assumption \ref{ass:FB} and \eqref{eq:norm_varphi_bounded_fluctuation_bound_proof}, we can bound 
\begin{align*}
\E\Big[ &|\langle g, \mu_t^N-\mu_t \rangle|^2 \ind_{\{t\le \tau^K_N\}} \Big] \\ &\le \; N^{-1} \operatorname{Var}_{\mu_0}(\varphi(0,\cdot))\; + \; 2L_K \int_0^t  \E\Big[ \lVert \mu^N_s-\mu_s \rVert_\ast^2 \ind_{\{s\le \tau^K_N\}} \Big] \, \diff s  \; + \; 2\E \big[\big(M^{N,\varphi}_t\big)^2 \big],
\end{align*} 
where we have used the Fubini-Tonelli theorem to exchange the time integral and the expectation. Moreover,
\begin{align}
\E\big[|\langle g, \mu_t^N-\mu_t \rangle|^2 \ind_{\{t\le \tau^K_N\}}\big]  \;
&\le \; 2L_K \int_0^t \E\Big[ \lVert\mu^N_s-\mu_s\rVert^2_{\ast} \ind_{\{t\le \tau^K_N\}} \Big] \, \diff s \; + \;\Oland(N^{-1}), \label{eq:inequality_fluctuation_bound_proof}
\end{align}
with $\Oland(N^{-1})$ independent of $s,t$ or $g, \varphi$. Indeed, the second inequality follows from the fact that the $W^i$ are independent and so $\E[ (M^{N,\varphi}_t)^2] = N^{-2} \sum_{i=1}^N \int_0^t \E[ \partial_x \varphi(s,X^{i,N}_s)^2 a(X^{i,N}_s, \mu^N_s)^2] \diff s \; = \;  N^{-1} \E[ \partial_x \varphi(s,X^{1,N}_s)^2 a(X^{1,N}_s, \mu^N_s)^2]$ by the exchangeability of the particles. Then, the quadratic growth of $a^2$ and the polynomial growth of $\partial_x \varphi$ of order $p$, together with the uniform moment bounds of $\mu^N$ (see \eqref{eq:moments_of_mu}; we only face moments up to order $2p+2$), yield that $\E[ (M^{N,\varphi}_t)^2]\le C_2/N$, where $ C_2>0$ is again independent of $g, \varphi$ or $t, N$. By monotone convergence, \eqref{eq:inequality_fluctuation_bound_proof} implies 
$$ \E\Big[\lVert \mu_t^N-\mu_t \rVert_\ast^2 \ind_{\{t\le \tau^K_N\}} \Big] \; \le \; 2L_K \int_0^t \E\Big[ \lVert\mu^N_s-\mu_s\rVert^2_{\ast} \ind_{\{t\le \tau^K_N\}} \Big] \, \diff s \; + \;\Oland(N^{-1}).$$
Note that the inequality holds for all $0\le t\le T$. Through an application of Grönwall's inequality, we finally deduce
$$ \operatorname{sup}_{0\le t \le T} \E\Big[ \lVert \mu^N_t - \mu_t \rVert_{\ast}^2 \ind_{\{t\le \tau^K_N\}} \Big] \; \le \; \frac{C_K}{N}$$
for some constant $C_K>0$.
\end{proof}

\begin{proof}[Proof of Proposition \ref{prop:empirical_measure_decomposition}]
    Let $\varphi \in \mathscr{A}$ and recall the definition of the operator $\mathcal{V}$ in \eqref{eq:definition_of_operator_mathcal_V}. In full analogy to the beginning of the proof of Proposition \ref{lem:fluctuation_bound}, we obtain
    $$ \langle \varphi(t,\cdot),(\mu^N_t-\mu_t)\rangle \; = \;   I^{N,\varphi}_t \; + \; \int_0^t \langle \partial_s \varphi(s,\cdot) + \mathcal{V}_{s,\mu^N_s}\varphi, \mu^N_s\rangle - \langle \partial_s \varphi(s,\cdot) + \mathcal{V}_{s,\mu_s}\varphi, \mu_s\rangle\, \diff s \; + \;  M^{N,\varphi}_t,$$
    where 
    $$I^{N,\varphi}_t:=\langle \varphi(0,\cdot), \, \mu^N_0-\mu_0 \rangle \qquad \text{ and } \qquad  M^{N,\varphi}_t := N^{-1} \sum_{i=1}^N \int_0^t \langle a(\cdot, \mu^N_s)  \partial_x \varphi(s,\cdot), \, \mu^N_s\rangle \diff W^i_s.$$ As described before \eqref{eq:backward_PDE_with_A_operator_Schauder}, $\nu \mapsto \langle \mathcal{V}_{s,\nu}\varphi, \nu \rangle $ is twice differentiable for each $s \ge 0$ at $\mu_s$, and so we expand 
    $$ \langle \mathcal{V}_{s,\mu^N_s}\varphi, \mu^N_s\rangle - \langle \mathcal{V}_{s,\mu_s}\varphi, \mu_s\rangle \; = \; \langle D_\nu \langle \mathcal{V}_{s,\mu_s}\varphi, \mu_s\rangle, \mu^N_s - \mu_s\rangle \; + \; R^{N, \varphi}_s,$$
    with \eqref{eq:final_decomposition_limit_behaviour_remainder} according to Corollary \ref{cor:bounding_the_remainder} and Conditions \ref{ass:FB}, \ref{ass:R2}, and \ref{condition:coefficient_regularity_extended_backward_PDE}. Thus, this yields 
    \begin{align}
        \langle \varphi(t,\cdot),(\mu^N_t-\mu_t)\rangle  \; = \;    I^{N,\varphi}_t \; + \; \int_0^t \langle \partial_s \varphi(s,\cdot) + \mathcal A_{s,\mu_s} \varphi, (\mu^N_s-\mu_s) \rangle \diff s + \int_0^t R^{N, \varphi}_s \diff s +  M^{N,\varphi}_t. \label{eq:decomp_before_PDE_argument_proof_empirical_measure_CLT}
    \end{align} 
Consider now the backward PDE \eqref{eq:backwards_PDE}. Due to Theorem \ref{thm:propagation_of_polynomial_control_extended_PDE_A}, \eqref{eq:backwards_PDE} admits a solution, for which we denote $\varphi=\varphi^{\phi,r} \in \mathscr{A}$ its extension (see Remark \ref{rem:extension_for_solutions_on_time_subintervals}). Replacing $\partial_s \varphi(s,\cdot) + \mathcal A_s \varphi$ by $-\phi$ in \eqref{eq:decomp_before_PDE_argument_proof_empirical_measure_CLT}, considering $t=r$, and using the terminal condition of \eqref{eq:backwards_PDE} yields 
    $$  \int_0^r \langle \phi(s,\cdot), (\mu^N_s-\mu_s) \rangle \diff s \; =  \;  \langle \varphi^{\phi,r}(0,\cdot), \, \mu^N_0-\mu_0 \rangle \; + \;   M^{N,\phi,r}_r \, + \, \int_0^r R^{N, \phi,r}_s \diff s $$
    where we have denoted the martingale $M^{N,\phi,r}:= M^{N,\varphi^{\phi,r}}$ and the remainder $R^{N,\phi,r}:=R^{N,\varphi^{\phi,r}}$ in order to emphasise the dependence on $\phi$ (and the terminal time $r$). Since the quantities $\lVert\varphi^{\phi,r}\rVert_{\mathscr{A}}$, $0\le r\le T$, are uniformly bounded, we deduce \eqref{eq:final_decomposition_limit_behaviour_remainder} from Corollary \ref{cor:bounding_the_remainder}.
\end{proof}

\begin{proof}[Proof of Lemma \ref{lem:quadratic_covariation_M_phi}] 
    Fix $0\le t\le T$ and let $0\le u,r\le T$. By independence of $W^i\indep W^j$, $i\neq j$, we obtain
    \begin{align}
         \Big\langle \sqrt{N} M^{N,\phi,u}, &\sqrt{N}M^{N, \tilde \phi,r}\Big\rangle_t \,
        = \; \frac{1}{N} \sum_{i=1}^N \int_0^{t\wedge u\wedge r} a(X^{i,N}_s, \mu^N_s)^2 \; \partial_x \varphi^{ \phi,u}(s,X^{i,N}_s)\partial_x \varphi^{ \tilde \phi,r}(s,X^{i,N}_s)\;  \, \diff s \notag \\
        &= \; \frac{1}{N} \sum_{i=1}^N \int_0^{t\wedge u\wedge r} a(X^{i}_s, \mu_s)^2 \; \partial_x \varphi^{ \phi,u}(s,X^{i}_s)\partial_x \varphi^{ \tilde \phi,r}(s,X^{i}_s)\, \diff s \; + \; R^{N,1}_t \; + \; R^{N,2}_t \label{eq:proof_martingale_CLT}
    \end{align}
    where the $X^i$ are synchronous couplings of the $X^{i,N}$ with respect to the McKean-Vlasov laws $(\mu_s)_{0\le s \le T}$ (that is driven by the same Brownian motions $W^i$), and 
   \begin{align*} 
   R^{N,1}_t \; &= \; \int_0^{t\wedge u\wedge r} \int_{\R} \big[a(x, \mu^N_s)^2-a(x, \mu_s)^2 \big] \; \partial_x \varphi^{ \phi,u}(s,x)\partial_x \varphi^{ \tilde \phi,r}(s,x) \,\mu^N_s(\diff x) \, \diff s, \\
   R^{N,2}_t \; &= \; \int_0^{t\wedge u\wedge r} \int_{\R} a(x, \mu_s)^2 \; \partial_x \varphi^{ \phi,u}(s,x)\partial_x \varphi^{ \tilde \phi,r}(s,x) \,(\mu^N_s-\bar{\mu}^N_s)(\diff x) \, \diff s,
   \end{align*}
    with $\bar{\mu}^N_s$ denoting the empirical law of the $X^i_s$. The first term on the right-hand side of \eqref{eq:proof_martingale_CLT} converges to 
    $$ \E\Big[ \int_0^{t\wedge u\wedge r} a(X_s, \mu_s)^2 \; \partial_x \varphi^{\phi,u}(s,X_s)\partial_x \varphi^{\tilde \phi,r}(s,X_s) \diff s \Big]$$
    almost surely according to the strong law of large numbers. Indeed, due to the polynomial growth of $a^2$ and the $\partial_x \varphi^{\phi,s}, \partial_x \varphi^{\tilde \phi,s}$, $0\le s\le T$, and the uniform moment bounds \eqref{eq:moments_of_mu}, the first moment of the integrals exists. 
    
    Now, due to \eqref{eq:moments_of_mu}, the random variables 
    $$\Gamma_N \; := \; \operatorname{max}_{q=1,...,2p+2} \big( \operatorname{sup}_{0\le t\le T} \mathcal{M}_{q}(\bar{\mu}^N_t) \, + \, \operatorname{sup}_{0\le t\le T} \mathcal{M}_{q}(\mu^N_t)\big), \qquad N\ge 1, $$ are uniformly integrable. Therefore, to every $\varepsilon>0$ we can find $R_\varepsilon>0$ such that 
    $$ \operatorname{sup}_{N\ge 1} \Pro( \Gamma_N > R_\varepsilon) \; \le \; \varepsilon.$$ 
    Let $\gamma>0$. Note that 
    \begin{align*} \Pro(|R^{N,1}_t| > \gamma) \; &\le \; \Pro \Big( C \int_0^t \mathcal{W}_p(\mu^N_s,\mu_s) N^{-1}\sum_{i=1}^N (1+|X^{i,N}_s|^2)^{p} (1+ |X^{i,N}_s|^2 + \mathcal{M}_p(\mu^N_s)^{2/p}) \diff s \, > \, \gamma \Big) \\
    &\le \; \Pro \Big( \Gamma_N \int_0^{t\wedge u\wedge r} \mathcal{W}_p(\mu^N_s,\mu_s) \diff s \, > \, \gamma/\tilde C \Big) \\
    & \le \; \varepsilon \, + \, \gamma^{-1} \tilde C R_\varepsilon \int_0^{t\wedge u\wedge r} \E\big[\mathcal{W}_p(\mu^N_s,\mu_s)\big] \diff s \; \le \; \varepsilon \; + \; \Oland(N^{-1/(2p)})
    \end{align*}
    for all $N\ge 1$ and some constants $C, \tilde C>0$, by virtue of the Lipschitz continuity of $a$, its growth control, the polynomial growth of the $\partial_x\varphi^{\phi,s},\partial_x\varphi^{\tilde \phi,s}$, as well as the Wasserstein rate bound given in Proposition \ref{prop:propagation_chaos_Wp}. Since $\varepsilon>0$ was chosen arbitrarily, we deduce $R^{N,1}_t \to 0$ in probability. We can make use of a similar argument in order to obtain that $R^{N,2}_t \to 0$ in probability. The maps $x\mapsto a(x, \mu_s)^2 \; \partial_x \varphi^{\phi,u}(s,x)\partial_x \varphi^{\tilde \phi,r}(s,x)$, $0\le s \le t$, are Lipschitz with a common Lipschitz constant on compacts. Restricted on a compact spatial interval $[-r,r]$ without loss of generality we can assume the functions to be uniformly Lipschitz, and the Kantorovich-Rubinstein theorem then yields that $\langle a(\cdot,\mu_s)^2\partial_x \varphi^{\phi,u}(s,\cdot)\partial_x \varphi^{\tilde \phi,r}(s,\cdot), \mu^N_s-\bar{\mu}^N_s\rangle$ is bounded by $W_1(\mu^N_s,\bar{\mu}^N_s)\le \mathcal{W}_p(\mu^N_s,\bar{\mu}^N_s)$, up to a constant uniformly in $s$. Now, proceed similarly by conditioning on the set $\{\Gamma_N\le R_\varepsilon\}$, then dividing the integrals into a part where the functions are restricted to a compact, and use Kantorovich-Rubinstein on the former while bounding the second suitably based on the moment condition $\{\Gamma_N\le R_\varepsilon\}$ as before.
\end{proof}

\begin{proof}[Proof of Theorem \ref{thm:functional_CLT_additive_fluctuations}] 
We begin by showing that the sequence $\Psi^N(\phi)$ is tight in the space $(\mathscr{C}([0,T];\R), \, \norm{\cdot}_\infty)$. In fact, due to \eqref{eq:stopping_time_vanishing_probabilities} it suffices to show that, for every $K>0$, the sequence $\Psi^{N,K}(\phi)$, $N\ge 1$, with
$$ \Psi^{N,K}(\phi)_t \; := \; \int_0^{t\wedge \tau^K_N} \langle \phi(s,\cdot), \sqrt{N}(\mu^N_s - \mu_s) \rangle \diff s, \qquad 0\le t\le T,$$
is tight in $(\mathscr{C}([0,T];\R), \, \norm{\cdot}_\infty)$. Fix $K>0$. We are going to apply the Kolmogorov-Chentsov tightness criterion (see \cite{billingsley}), stating that if we can find $\lambda,\beta>0$ and $C>0$ such that 
\begin{align}\E\Big[ \big|\Psi^{N,K}(\phi)_t -\Psi^{N,K}(\phi)_s \big|^\lambda \Big]  \; \le \; C |t-s|^{1+\beta} \label{eq:Kolmogorov_Chentsov_criterion} \end{align}
for all $0\le s,t\le T$ and $N\ge 1$, then this yields the desired tightness. Indeed, to see this make use of the Cauchy-Schwarz' inequality, the Fubini-Tonelli theorem, and the fact that $|\langle \phi(r,\cdot), \mu^N_r-\mu_r\rangle|^2\lesssim \lVert \mu^N_r-\mu_r\rVert_*^2$, which is valid since $\phi \in \mathscr{A}$, in order to obtain
\begin{align*}
    \E\Big[ \big|\Psi^{N,K}(\phi)_t -\Psi^{N,K}(\phi)_s \big|^2 \Big] \; &\le \; |t-s| \, \E\Big[  \int_{s\wedge \tau^K_N}^{t\wedge \tau^K_N} |\langle \phi(s,\cdot), \sqrt{N}(\mu^N_s - \mu_s) \rangle|^2 \diff s  \Big] \\
    &\le \; |t-s| \, N\int_{s\wedge \tau^K_N}^{t\wedge \tau^K_N} \E\Big[ |\langle \phi(s,\cdot), (\mu^N_s - \mu_s) \rangle|^2 \Big] \diff s  \\
    &\le \; |t-s|^2 \, N\operatorname{sup}_{0\le r\le T} \E\Big[ \lVert \mu^N_r-\mu_r\rVert_*^2 \ind_{\{r\le \tau^K_N\}} \Big].
\end{align*}
Finally, Proposition \ref{lem:fluctuation_bound} yields \eqref{eq:Kolmogorov_Chentsov_criterion}, and we deduce the tightness of the sequence $(\Psi^N(\phi))_{N\ge 1}$.

In a next step, we shall prove the convergence of finite-dimensional distributions. Classically, this then implies the desired weak convergence. First, note that, due to \eqref{eq:stopping_time_vanishing_probabilities} and \eqref{eq:final_decomposition_limit_behaviour_remainder}, the third term on the right-hand side of \eqref{eq:decomposition_additive_fluctuation} converges to zero in probability. First we show that for any $q\ge 1$, $(t_1,...t_q) \in [0,T]^q$ we have
$$ \sqrt{N}\Big(M^{N,\phi,t_1}_{t_1},...,M^{N,\phi,t_q}_{t_q}\Big) \; \conv{N\to \infty}{\mathcal{L}} \; (\tilde Z_{t_1},...,\tilde Z_{t_q})$$
for some centred Gaussian random variables $\tilde Z_{t_i}$ satisfying $\operatorname{Cov}(\tilde Z_{t_i}, \tilde Z_{t_j})=\Lambda_{t_i,t_j}^{\phi,\phi}(t_i\wedge t_j)$. Let $q\ge 1$ and $(t_1,...t_q) \in [0,T]^q$. According to Rebolledo's convergence theorem \cite[Prop.~1]{rebolledo} and Lemma \ref{lem:quadratic_covariation_M_phi}, the sequence of $q$-dimensional martingales $\sqrt{N}(M^{N,\phi,t_1},...,M^{N,\phi,t_q})$ converges to a continuous multivariate centred Gaussian martingale with covariance function $\Lambda^{\phi,\phi}_{t_i,t_j}(s)$, $i,j=1,...,q$, $0\le s\le T$, on the Skorokhod space hence in $(\mathscr{C}([0,T];\R^d), \, \norm{\cdot}_\infty)$. In particular, since the limit is continuous, coordinate projections are weakly continuous and we deduce the claimed convergence of finite-dimensional distributions. Moreover, by the classical central limit theorem, it holds that the finite-dimensional distributions of the process $r\mapsto \sqrt{N} \langle \varphi^{\phi,r}(0,\cdot),\mu^N_0-\mu_0\rangle=: I^{N,\phi}_r$ converge to multivariate centred Gaussian vectors with covariance $\operatorname{Cov}_{\mu_0}(\varphi^{\phi,t_i}(0,\cdot), \varphi^{\phi,t_j}(0,\cdot))$. Since the $M^{N,\phi,t_i}_{t_i}$ and the $I^{N,\phi}_{t_j}$ are mutually independent, the claim follows.
\end{proof}


\subsection{Proofs to Section \ref{sec:Stat}}
Before we are in a position to prove Theorem \ref{thm:CLT_for_estimators}, we will establish an intermediate decomposition given in Proposition \ref{prop:martingale_decomp} below. Toward this, for measurable functions $f,g:\R\times\mathcal P_p(\R)\to\R$ define
$$
F_{(f,g)}(\nu)
:=
\langle f(\cdot,\nu)g(\cdot,\nu),\nu\rangle.
$$
and 
$$ \phi_{(f,g)}(t,x) \; := \; D_\nu F_{(f,g)}(\mu_t)(x).$$
Then,
$
\Psi(f,g)
=
\int_0^T F_{(f,g)}(\mu_t)\diff t.
$ Further, set 
$$ Q^{N,j}_{t} :=  \frac 1 N \sum_{i=1}^N \int_0^t b_j(X^{i,N}_s, \mu^N_s) a(X^{i,N}_s, \mu^N_s) \, \diff W^i_s, \qquad j=1,...,d.$$

\begin{prop}[Approximate martingale-type decomposition of estimators] \label{prop:martingale_decomp}
    Under the regularity Conditions \ref{ass:R}, \ref{condition:coefficient_regularity_extended_backward_PDE} and Assumptions \ref{ass:FB}--\ref{ass:R3}, the estimators given in Section \ref{subsec:parametric_testing} decompose as
    \begin{align*}
        &\sqrt{N}\big(\widehat{\Psi}(b_j,b_k)_N- \Psi(b_j,b_k)\big) \; = \; \sqrt{N}\Big(\langle \varphi^{\phi_{(b_j,b_k)},T}(0,\cdot), \, \mu^N_0-\mu_0 \rangle \, + \,  M^{N, \, \phi_{(b_j,b_k)},T}_T\Big) \; + \; \oland_{\Pro}(1) ,\\
        &\sqrt{N}\big(\widehat{\Psi}(b_j,b)_N - \Psi(b_j,b)\big)\, = \, \sqrt{N} \Big(\sqrt{N}\langle \varphi^{\phi_{(b_j,b)},T}(0,\cdot), \, \mu^N_0-\mu_0 \rangle \, + \,  Q^{N,j}_{T}\, + \, M^{N, \, \phi_{(b_j,b)},T}_T\Big)+  \oland_{\Pro}(1),
    \end{align*} 
    where $\oland_{\Pro}(1)$ shall be understood with respect to convergence of $N\to \infty$.
\end{prop}
\begin{proof}[Proof of Proposition \ref{prop:martingale_decomp}] 
    For the estimators $\widehat{\Psi}(b_j,b_k)_N$, due to (ii) of Condition \ref{ass:R3} and Remark \ref{rem:remark_to_condition_R1}, we can perform an expansion \eqref{eq:Taylor_expansion_measure} in Corollary \ref{cor:bounding_the_remainder} with $\Phi(\varphi,t,\nu)=F_{(b_j,b_k)}(\nu)$, in order to obtain 
    \begin{align*}
        \sqrt{N} \big(\widehat{\Psi}(b_j,b_k)_N- \Psi(b_j,b_k)\big) \; &=\; \sqrt{N} \int_0^T F_{(b_j,b_k)}(\mu^N_t) - F_{(b_j,b_k)}(\mu_t)\, \diff t \\
        &= \; \int_0^T \langle D_\nu (F_{(b_j,b_k)}(\mu_t)), \sqrt{N}(\mu^N_t-\mu_t) \rangle \, \diff t \; + \; \oland_{\Pro}(1).
    \end{align*}
    Now, apply Proposition \ref{prop:empirical_measure_decomposition} with $\phi(t,x)=D_\nu (F_{(b_j,b_k)}(\mu_t))(x)$ and $t=T$. This yields the desired decomposition.
    
    Secondly, the SDE definition of $X^{i,N}$ in \eqref{eq:IPS} gives 
    $$ \sqrt{N}\big(\widehat{\Psi}(b_j,b)_N - \Psi(b_j,b)\big) \; = \; \sqrt{N} \, Q_T^{N} \; + \; \sqrt{N} \int_0^T F_{(b_j,b)}(\mu^N_t) - F_{(b_j,b)}(\mu_t) \, \diff t,$$
    and proceeding as for $\widehat{\Psi}(b_j,b_k)_N$ above yields the analogous decomposition for this part.
\end{proof}

\begin{remark}
    The unspecified convergence rates $\oland_{\Pro}(1)$ in Proposition \ref{prop:martingale_decomp} can be made precise if we have more information on the convergence rates in Definition \ref{defi:CLT_fluctuation_compatibility} and in  \eqref{eq:stopping_time_vanishing_probabilities}. Indeed, if, for example, $N^{1/2}c_N(\varphi) \vee N^{-1/2}C_N(\varphi)=\Oland(N^{-\varepsilon})$ for some $\varepsilon >0$, and \eqref{eq:stopping_time_vanishing_probabilities} vanishes at the same rate, then the convergence rate of the $\oland_{\Pro}(1)$ in Proposition \ref{prop:martingale_decomp} will also be of order $\Oland(N^{-\varepsilon})$.
\end{remark}

We can now prove Theorem \ref{thm:CLT_for_estimators}.
\begin{proof}[Proof of Theorem \ref{thm:CLT_for_estimators}]
   Based on Proposition \ref{prop:martingale_decomp}, one first shows an analogous result to Lemma \ref{lem:quadratic_covariation_M_phi} accommodating for the martingale term $Q^{N,j}+M^{N,\phi_{(b_j,b)},T}$. Then Theorem \ref{thm:CLT_for_estimators} can be proved along the same lines as for the convergence of finite-dimensional distributions in the second part of the proof of Theorem \ref{thm:functional_CLT_additive_fluctuations}.  
\end{proof}

Finally, we provide the proof of Theorem \ref{thm:CLT_for_projection_test}.
\begin{proof}[Proof of Theorem \ref{thm:CLT_for_projection_test}]
By Theorem~\ref{thm:CLT_for_estimators},
\[
\sqrt{N}
\begin{bmatrix}
\widehat{{\bf \Gamma}}_N-{\bf \Gamma}\\
\operatorname{vec}\big(\widehat{{\bf \Pi}}_N-{\bf \Pi}\big)
\end{bmatrix}
\; \conv{N\to\infty}{\mathcal L} \; 
\Norm_{d+d^2}({\bf 0},{\bf \Sigma}).
\]
In particular,
\[
\widehat{{\bf \Pi}}_N \; \conv{N\to\infty}{\Pro} \; {\bf \Pi}.
\]
Since ${\bf \Pi}$ is positive definite, it is invertible. By continuity of the smallest singular value, $
\Pro(\widehat{{\bf \Pi}}_N\in\operatorname{GL}_d(\R)) \rightarrow 1$.
Consequently, $\widehat{{\bf \Pi}}_N^\dagger=\widehat{{\bf \Pi}}_N^{-1}$ with probability tending to one, so it suffices to derive the limiting distribution using the ordinary inverse. Consider the map
\[
G:\R^d\times\operatorname{GL}_d(\R)\longrightarrow\R^r,
\qquad
G({\bf y},{\bf E})
:=
\Xi{\bf E}^{-1}{\bf y}.
\]
The map $G$ is continuously differentiable, and for
\({\bf h}\in\R^d\) and \({\bf H}\in\R^{d\times d}\), its differential is
\[
DG_{({\bf y},{\bf E})}[{\bf h},{\bf H}]
\, = \, 
\Xi{\bf E}^{-1}{\bf h}
-
\Xi{\bf E}^{-1}{\bf H}{\bf E}^{-1}{\bf y}.
\]
Since \({\bf \Pi}^{-1}{\bf \Gamma}=\theta^*\), it follows that
\[
DG_{({\bf \Gamma},{\bf \Pi})}[{\bf h},{\bf H}]
 \, = \, 
\Xi{\bf \Pi}^{-1}{\bf h}
-
\Xi{\bf \Pi}^{-1}{\bf H}\theta^*.
\]
If we denote $I_{d}$ the $d\times d$-identity matrix and $\otimes$ the Kronecker product, then
\[
{\bf H}\theta^*
=
\big((\theta^*)^\top\otimes I_{d}\big)\operatorname{vec}({\bf H}),
\]
and hence
\[
DG_{({\bf \Gamma},{\bf \Pi})}[{\bf h},{\bf H}]
 \; = \; 
J_G
\begin{bmatrix}
{\bf h}\\
\operatorname{vec}({\bf H})
\end{bmatrix},
\]
where
\[
J_G
:=
\left[
\Xi{\bf \Pi}^{-1},
\;
-(\theta^*)^\top\otimes\big(\Xi{\bf \Pi}^{-1}\big)
\right]
\in\R^{r\times(d+d^2)}.
\]
Indeed, by the Kronecker product identity, it holds that $\Xi \Pi^{-1}((\theta^*)^\top\otimes I_{d})=(\theta^*)^\top\otimes (\Xi \Pi^{-1})$. Applying the multivariate delta method to Theorem~\ref{thm:CLT_for_estimators} finally yields
\[
\sqrt{N}
\Big(
G(\widehat{{\bf \Gamma}}_N,\widehat{{\bf \Pi}}_N)
-
G({\bf \Gamma},{\bf \Pi})
\Big)
\; \conv{N\to\infty}{\mathcal L} \; 
\Norm_r
\Big(
{\bf 0},
J_G{\bf \Sigma}J_G^\top
\Big).
\]
\end{proof}

\begin{appendix}
\section{Appendix}\label{appn} 

\subsection{Uniform moment bounds and propagation of chaos}
\label{sec:appendix_propagation_of_chaos_Wp}
Throughout this section, we assume the fundamental Assumption \ref{ass:Main} to hold.

\begin{prop}[Moment propagation]
\label{prop:Wp_well_posedness_moment_propagation}
The $p$-th moments of the McKean--Vlasov SDE \eqref{eq:MKV} as well as the associated interacting particle system \eqref{eq:IPS} are uniformly bounded in the sense 
\begin{align}
\E[ \operatorname{sup}_{0\le t\le T} |X_t|^p]
\; + \; 
\operatorname{sup}_{N\ge1}
\E\big[ \operatorname{sup}_{0\le t\le T} \mathcal{M}_p(\mu_t^N)\big]
<\infty. \label{eq:appendix_moment_bound}
\end{align}
\end{prop}

\begin{proof}
Let $Y_t^N:=\mathcal{M}_p(\mu_t^N)$. Applying Itô's formula to $|X_t^{i,N}|^p$ gives, up to the usual smooth approximation at the origin,
\begin{align*}
\diff |X_t^{i,N}|^p
&=
p|X_t^{i,N}|^{p-2}X_t^{i,N}b(X_t^{i,N},\mu_t^N)\diff t \\
&\quad+
\frac{p(p-1)}2
|X_t^{i,N}|^{p-2}a(X_t^{i,N},\mu_t^N)^2\diff t
+
\diff M_t^{i,N},
\end{align*}
where $M^{i,N}$ is a local martingale. By the linear growth \eqref{eq:linear_growth} and Young's inequality,
$$
p|X^{i,N}_t|^{p-1}|b(X^{i,N}_t,\mu^N_t)|
+
\frac{p(p-1)}{2}|X^{i,N}_t|^{p-2}|a(X^{i,N}_t,\mu^N_t)|^2
\le
C\left(1+|X^{i,N}_t|^p+\mathcal{M}_p(\mu^N_t)\right).
$$
for some new constant $C>0$ independent of $i,N$, and $t$.
Averaging over $i=1,\ldots,N$ and taking expectations yields
$$
\diff \E[Y_t^N]
\le
2C\left(1+\E[Y_t^N]\right) \diff t.
$$
Since $\E[Y_0^N]=\mathcal{M}_p(\mu_0)< \infty$, Gronwall's lemma gives
$$
\operatorname{sup}_{N\ge1}
\operatorname{sup}_{0\le t\le T}
\E[\mathcal{M}_p(\mu_t^N)]
<\infty .
$$
Indeed, the continuity $t\mapsto \E[Y^N_t]$ follows from the continuity of the suitably stopped processes and Fatou's lemma. Now, with this at hand, proceed with a similar argument considering $\operatorname{sup}_{0\le s\le t} |Y^{N}_s|$. Then, apply the Burkholder-Davis-Gundy inequality to $\operatorname{sup}_{0\le s\le t} |N^{-1} \sum_{i=1}^N M^{i,N}_s|$, make use, once again, of the growth bound and Young's inequality, and exploit the already proven weaker moment bounds in order to deduce
$$\operatorname{sup}_{N\ge1}
\E\big[ \operatorname{sup}_{0\le t\le T} \mathcal{M}_p(\mu_t^N)\big]< \infty.$$

The moment estimate for the McKean--Vlasov solution follows analogously.
\end{proof}

\begin{remark} \label{rem:moment_propagation}
    Note that the moment propagation in Proposition \ref{prop:Wp_well_posedness_moment_propagation} is true for any moment $q\ge p$ if $\mathcal{M}_q(\mu_0)<\infty$.
\end{remark}

\begin{remark} \label{rem:refined_explanation_of_moment_propagation}
    A careful inspection of the proof of Proposition \ref{prop:Wp_well_posedness_moment_propagation} shows that, in particular, we have $\operatorname{sup}_{0\le t\le T} \E[|X_t|^p]\le C\E[|X_0|^p]$ for some constant $C>0$ only depending on $T$ and $p$. 
\end{remark}

\begin{lem}[Empirical convergence of independent McKean--Vlasov copies]
\label{lem:iid_empirical_Wp}
For every $N\ge 1$, let $\bar X^{1,N},\bar X^{2,N},\ldots, \bar X^{N,N}$ be independent copies of the solution to the McKean--Vlasov SDE \eqref{eq:MKV}, and define
$$
\bar\mu_t^N
:=
\frac1N\sum_{j=1}^N\delta_{\bar X_t^{j,N}}.
$$
Then, as $N\to \infty$,
\begin{align*}
\E\big[
\operatorname{sup}_{0\le t\le T}
\mathcal{W}_p(\bar\mu_t^N,\mu_t)^p
\big] \; 
\to \; 0.
\end{align*}
\end{lem}
\begin{proof}
Let $Q$ denote the law of the McKean--Vlasov path $(X_t)_{0\le t\le T}$ on $\mathscr{C}([0,T];\mathbb R)$, equipped with its uniform norm $\norm{\cdot}_\infty$.
Let
$$
\bar Q^N
:=
\frac1N\sum_{j=1}^N\delta_{\bar X^{j,N}}
$$
be the empirical measure of the independent paths. By Proposition \ref{prop:Wp_well_posedness_moment_propagation},
$$
\int_{\mathscr{C}([0,T];\mathbb R)}
\lVert \lambda \rVert_\infty^p \, Q(\diff\lambda)
=
\E\left[
\operatorname{sup}_{0\le t\le T}|X_t|^p
\right]
<\infty.
$$
The (triangular array) law of large numbers gives that $\bar Q^N$ converges weakly to $Q$ almost surely, and it also yields the convergence of the $p$-th moments. Hence, it is well-known that this implies convergence of $\mathcal{W}_p(\bar Q^N, Q) \to 0$ almost surely and in $L^p$, where $\mathcal{W}_p$ is used here to denote the Wasserstein-$p$ distance on the Polish space $(\mathscr{C}, \norm{\cdot }_\infty)$.
For each $0\le t \le T$, define the coordinate evaluation map
$$
\pi_t:\mathscr{C}([0,T];\mathbb R)\to\mathbb R,
\qquad
\pi_t(\lambda)=\lambda_t.
$$
which is Lipschitz. Thus, 
$$
\mathcal{W}_p(\bar\mu_t^N,\mu_t)
=
\mathcal{W}_p(\bar Q^N \circ \pi_t^{-1}, Q \circ \pi_t^{-1})
\lesssim
\mathcal{W}_p(\bar Q^N,Q).
$$
Taking the supremum over $0\le t\le T$ gives
$$
\operatorname{sup}_{0\le t\le T}
\mathcal{W}_p(\bar\mu_t^N,\mu_t)
\lesssim
\mathcal{W}_p(\bar Q^N,Q).
$$
\end{proof}

\begin{lem}[Synchronous coupling estimate]
\label{lem:synchronous_coupling_Wp}
Let $\bar X^{i,N}$, $i=1,...,N$, be (independent) solutions to the McKean--Vlasov SDE \eqref{eq:MKV}, each driven by the same Brownian motions $W^i$ as the particles and starting from the same initial condition:
$$
\diff \bar X_t^i
=
b(\bar X_t^i,\mu_t)\diff t
+
a(\bar X_t^i,\mu_t)\diff W_t^i,
\qquad
\bar X_0^{i,N}= X_0^{i,N}.
$$
Then, it holds that
\begin{align*}
\E\Big[
\operatorname{sup}_{0\le t\le T}
|X_t^{i,N}-\bar X_t^i|^p
\Big]
\; \to \; 0.
\end{align*}
\end{lem}
\begin{proof}
Set $\Delta_t^{i,N}:=X_t^{i,N}-\bar X_t^i$, and observe
\begin{align*}
\Delta_t^{i,N}
=
\int_0^t
\Big[
b(X_s^{i,N},\mu_s^N)-b(\bar X_s^i,\mu_s)
\Big]\diff s 
\quad+
\int_0^t
\Big[
a(X_s^{i,N},\mu_s^N)-a(\bar X_s^i,\mu_s)
\Big]\diff W_s^i .
\end{align*}
Making use of the Lipschitz continuity of $b$ and $a$, we have
$$
|b(X_s^{i,N},\mu_s^N)-b(\bar X_s^i,\mu_s)|
+
|a(X_s^{i,N},\mu_s^N)-a(\bar X_s^i,\mu_s)|
\le
L\big(
|\Delta_s^{i,N}|+\mathcal{W}_p(\mu_s^N,\mu_s)
\big).
$$
for all $i=1,...,N$ and $0\le s \le T$. Using Burkholder--Davis--Gundy's inequality, there exists a constant $C>0$ such that
\begin{align*}
\E\left[
\operatorname{sup}_{u\le t}|\Delta_u^{i,N}|^p
\right]
&\lesssim
\int_0^t
\E\left[
|\Delta_s^{i,N}|^p
+
\mathcal{W}_p(\mu_s^N,\mu_s)^p
\right]\diff s\\
&\lesssim
\int_0^t
\E\left[
|\Delta_s^{i,N}|^p
+
\mathcal{W}_p(\mu_s^N,\bar\mu_s^N)^p
+
\mathcal{W}_p(\bar\mu_s^N,\mu_s)^p
\right]\diff s.
\end{align*}
Clearly,
$$
\E\Big[
\mathcal{W}_p(\mu_s^N,\bar\mu_s^N)^p
\Big]
\le
\frac1N\sum_{j=1}^N \E\big[|\Delta_s^{j,N}|^p\big] = \E\big[|\Delta_s^{i,N}|^p\big]
$$
by exchangeability of the particles. This implies
\begin{align*}
\E\left[
\operatorname{sup}_{u\le t}|\Delta_u^{i,N}|^p
\right]
&\lesssim 
\int_0^t
\E\big[
\operatorname{sup}_{u\le s} |\Delta_s^{i,N}|^p \big] \diff s \, + \, \int_0^t
\E\big[
\mathcal{W}_p(\bar\mu_s^N,\mu_s)^p \big] \diff s.
\end{align*}
Applying Gronwall's lemma yields
$$
\E\left[
\operatorname{sup}_{u\le t}|\Delta_u^{i,N}|^p
\right]
\le
C_T
\int_0^T
\E\left[
\mathcal{W}_p(\bar\mu_s^N,\mu_s)^p
\right]\diff s.
$$
By virtue of Lemma \ref{lem:iid_empirical_Wp}, we deduce the claim.
\end{proof}


\begin{prop}[Propagation of chaos in $\mathcal{W}_p^p$]
\label{prop:propagation_chaos_Wp}
It holds that
\begin{align*}
\E\Big[
\operatorname{sup}_{0\le t\le T}
\mathcal{W}_p(\mu_t^N,\mu_t)^p
\Big]
\; \lesssim  \;  N^{-1/2} .
\end{align*}
Moreover, for every fixed $k\ge1$,
$$
(X^{1,N},\ldots,X^{k,N})
\conv{}{\mathcal{L}}
(X^1,\ldots,X^k)
$$
in $\mathscr{C}([0,T],\mathbb R)^k$, where $X^1,\ldots,X^k$ are independent copies of the solution to the McKean--Vlasov SDE \eqref{eq:MKV}.
\end{prop}

\begin{proof}
Let $\bar X^1,\ldots,\bar X^N$ and $\bar\mu_t^N$ be as in Lemma~\ref{lem:synchronous_coupling_Wp}. Observe that 
$$
 \operatorname{sup}_{0\le t\le T}
\mathcal{W}_p(\mu_t^N,\bar\mu_t^N)^p
\le
\frac1N\sum_{j=1}^N
\operatorname{sup}_{0\le t\le T}
|X_t^{j,N}-\bar X_t^j|^p.
$$
Making use of the same procedure as in the proof of Lemma \ref{lem:synchronous_coupling_Wp} as well as \cite[Thm.~1]{FournierGuillin} (with, e.g., q=2p+3) yields the quantitative bound
    $$ \E\Big[\operatorname{sup}_{0\le t\le T} 
\mathcal{W}_p(\mu_t^N,\bar\mu_t^N)^p
\Big] \; \le \; \E\Big[
\operatorname{sup}_{0\le t\le T}
|X_t^{1,N}-\bar X_t^1|^p
\Big] \; \lesssim \; N^{-1/2},$$
and, hence, by triangle inequality and the fundamental inequality $|x+y|^p\le 2^p (|x|^p+|y|^p)$ we deduce the claim.
The finite-dimensional weak convergence in $\mathscr{C}$ follows in a straightforward way from Lemma \ref{lem:synchronous_coupling_Wp}.
\end{proof}

\subsection{Inequalities for mollifications}
Let $\varrho \in \mathscr{C}_c^\infty(\R)$ be a nonnegative and symmetric mollifier with $\operatorname{supp}(\varrho)\subset[-1,1]$ and $\int \varrho(z) \diff z = 1$, and define 
$$\varrho_\delta(x):= \delta^{-1} \varrho(x/\delta), \qquad \delta >0.$$
Furthermore, let $\chi_R \in \mathscr{C}^\infty(\R)$, $R>1$, be a smooth cut-off function such that $\chi_R\equiv 1$ on $[-2R,2R]$ and $\chi_R\equiv 0$ outside $[-3R, 3R]$, and with $\gamma:=\operatorname{sup}_{R>1}\operatorname{max}_{i=0,1,2} \lVert \partial_x^{(i)} \chi_R\rVert_\infty < \infty$.
For any $\varphi \in \mathscr{H}^{2+\alpha}_{P(r)}$, define the convolution
$$ \bar{\varphi}^{R,\delta}(x) := \; ((\varphi\chi_R) \star \varrho_\delta) \; = \; \int_{\R} \varphi(x-z) \chi_R(x-z) \varrho_\delta(z) \diff z.$$

\begin{lem} \label{lem:mollification_lemma_control_of_derivatives}
    For every integer $k\ge 0$, there exists $C>0$, independent of $\varphi$, $R$ and $\delta$, such that for every $\varphi \in \mathscr{H}^{2+\alpha}_{P(r)}$, $R>1$ and $0<\delta \le 1$,
    $$ \Big| \frac{\partial_x^{(k)} \bar{\varphi}^{R,\delta}(x)}{P(r)(x)} \Big| \; \le \; C \delta^{-((k-2)\vee 0)} \norm{\varphi}_{\mathscr{H}^{2+\alpha}_{P(r)}}. $$
\end{lem}
\begin{proof}
    Set $g_R:=\varphi\,\chi_R$. Then $\bar\varphi^{R,\delta}=g_R\star\varrho_\delta$. We first recall that there exists $C_r>0$ such that for all $x,z\in\R$,
        $$ \frac{P(r)(x-z)}{P(r)(x)} \;\le\; C_r\,P(r)(z).$$
Indeed, since $P(r)(x)=(1+x^2)^{r/2}$, we have
$$1+(x-z)^2\le 2(1+x^2)(1+z^2),$$
which implies $P(r)(x-z)\le 2^{r/2}P(r)(x)P(r)(z)$.
Since, in particular, $g_R\in \mathscr{C}^{2}(\R)$ and $\varrho_\delta\in \mathscr{C}^\infty_c(\R)$, standard arguments using dominated convergence and integration by parts yield that 
    $$ \partial_x^{(k)}\bar\varphi^{R,\delta} \;=\; \partial_x^{(k)}(g_R\star\varrho_\delta) \;=\; \begin{cases}
        (\partial_x^{(k)}g_R)\star \varrho_\delta, & k\le 2,\\[2pt]
        (\partial_x^{(2)}g_R)\star (\partial_x^{(k-2)} \varrho_\delta), & k\ge 3.
    \end{cases}$$
Then, for $k\ge 0$ we deduce
\begin{align*}
    \Big| \frac{\partial_x^{(k)} \bar{\varphi}^{R,\delta}(x)}{P(r)(x)} \Big| \; &\le \; C_r \Big(\frac{|\partial_x^{(k\wedge 2)}g_R|}{P(r)}\Big)\star (P(r) |\partial_x^{((k-2)\vee 0)} \varrho_\delta|) \\
    &\le \; C_r \Big(\sum_{\ell=0}^{k\wedge 2} \binom{k\wedge 2}{\ell}\,\frac{|\partial_x^{(\ell)}\varphi|}{P(r)}\,|\partial_x^{((k\wedge 2)-\ell)}\chi_R|\Big)\star (P(r) |\partial_x^{((k-2)\vee 0)} \varrho_\delta|)\\
    &\le \; C_r\gamma \norm{\varphi}_{\mathscr{H}^{2+\alpha}_{P(r)}} \int_{\R} P(r)(z) |\partial_x^{((k-2)\vee 0)} \varrho_\delta(z)| \diff z\\
        &\le \; C_r\gamma \norm{\varphi}_{\mathscr{H}^{2+\alpha}_{P(r)}} \delta^{-((k-2)\vee 0)} \int_{\R} P(r)(z) |\partial_x^{((k-2)\vee 0)} \varrho(z)| \diff z,
\end{align*}
The last inequality uses $\delta\le 1$ and $P(r)(\delta z)\le P(r)(z)$. This yields the claim.
\end{proof}

\begin{lem} \label{lem:mollificiation_lemma_approximation}
    There exists $C>0$, independent of $\varphi$, $R$ and $\delta$, such that for every $\varphi \in \mathscr{H}^{2+\alpha}_{P(r)}$, $0<\delta\le 1$ and $R>1\vee \delta$,
    $$  \operatorname{max}_{k=0,1} \, \operatorname{sup}_{|x|\le R}  \, \big| \partial_x^{(k)}\bar{\varphi}^{R,\delta}(x) - \partial_x^{(k)} \varphi(x) \big| \; \le \; C \delta (1+R^r) \norm{\varphi}_{\mathscr{H}^{2+\alpha}_{P(r)}}, $$
    and
    $$ \operatorname{sup}_{|x|\le R} \, \big| \partial_x^{(2)}\bar{\varphi}^{R,\delta}(x) - \partial_x^{(2)}\varphi(x) \big| \; \le \; C \delta^\alpha (1+R^r) \norm{\varphi}_{\mathscr{H}^{2+\alpha}_{P(r)}}. $$
\end{lem}
\begin{proof}
   Fix $0<\delta\le 1$, $R>1\vee \delta$ and let $x\in[-R,R]$, $k=0,1,2$. Write
    \begin{align}
        \partial_x^{(k)} \bar\varphi^{R,\delta}(x)-\partial_x^{(k)}\varphi(x) \; &=\;  \int_{\R} \varrho_\delta(z)\,\big(\partial_x^{(k)} (\varphi \chi_R)(x-z)-\partial_x^{(k)} \varphi(x)\big)\,dz \label{eq:appendix_1} 
\end{align}
By the choice of $\varrho$, $\operatorname{supp}(\varrho_\delta)\subset[-\delta,\delta]$. For $|z|\le\delta\le R$ we have $|x-z|\le 2R$, and by construction $\chi_R\equiv 1$ on $[-2R,2R]$. Therefore, $\varphi \chi_R (x-z)=\varphi(x-z)$ for all $z\in\operatorname{supp}(\varrho_\delta)$, hence $\partial_x^{(k)} (\varphi \chi_R)(x-z)=\partial_x^{(k)} \varphi(x-z)$, and we can further rewrite \eqref{eq:appendix_1} as
    \begin{align*}
        \partial_x^{(k)} \bar\varphi^{R,\delta}(x)-\partial_x^{(k)}\varphi(x) \; = \; \int_{\R} \varrho_\delta(z)\,\big(\partial_x^{(k)} \varphi(x-z)-\partial_x^{(k)} \varphi(x)\big)\,dz.
\end{align*}
Now, for $k=0,1$, by the mean value theorem and the weighted derivative bounds for $\varphi$,
\begin{align*}
|\partial_x^{(k)} \bar\varphi^{R,\delta}(x)-\partial_x^{(k)}\varphi(x)| \; &\le \; C_r (1+R^r) \, \norm{\varphi}_{\mathscr{H}^{2+\alpha}_{P(r)}} \int_{\R} \varrho_\delta(z)\,|z| \,dz \\
&\le \; C_r \delta (1+R^r) \, \norm{\varphi}_{\mathscr{H}^{2+\alpha}_{P(r)}} \int_{\R} \varrho(z)\,|z|  \,dz.
\end{align*}
Similarly, for $k=2$, the equivalent weighted Hölder norm gives
$$ |\partial_x^{(2)}\varphi(x-z)-\partial_x^{(2)}\varphi(x)|
\; \le \; C_r(1+R^r)\norm{\varphi}_{\mathscr{H}^{2+\alpha}_{P(r)}}|z|^\alpha,
\qquad |x|\le R,\quad |z|\le\delta. $$
Consequently, by scaling,
\begin{align*}
|\partial_x^{(2)} \bar\varphi^{R,\delta}(x)-\partial_x^{(2)}\varphi(x)| \; \le \; C_r \delta^\alpha (1+R^r) \, \norm{\varphi}_{\mathscr{H}^{2+\alpha}_{P(r)}} \int_{\R} \varrho(z)\,|z|^\alpha  \,dz.
\end{align*}
\end{proof}

\end{appendix}

\printbibliography

\end{document}